\documentclass{article}

\usepackage[latin2]{inputenc}
\usepackage[ngerman,english]{babel}		%am 10.11.2014 fuer das Titelblatt herausgenommen

\usepackage{amsthm}
\usepackage{amssymb}
\usepackage{amsmath}

\usepackage[nice]{nicefrac}

\usepackage{mathrsfs} %18.12.2012 damit man sehr geschwungene Buchstaben erzeugen kann, hier Schwarzraum

\usepackage{pstricks,pst-plot} 
\usepackage{tikz} %14.05.2012
\usetikzlibrary{decorations.pathreplacing,decorations.pathmorphing,calc} % 28.05.2013 fuer das bild am ende von F1 small

\usepackage{graphicx}
\usepackage{enumerate}	%25.04.2013

\usepackage{mathbbol}
\usepackage[nointegrals]{wasysym}
\usepackage{makeidx}
\usepackage{idxlayout} % 28.08.2013 fuert dazu, dass die Indexseite auch eine Kopfzeile hat
\makeindex

\usepackage{dlfltxbcodetips }%16.04.2012

\usepackage[linktocpage]{hyperref} % sollte immer als letztes Packet geladen

\usepackage{bbm}
\usepackage{caption}
\usepackage{mathtools}

\newcommand{\R}{\ensuremath{\mathbb{R}}}
\newcommand{\Z}{\ensuremath{\mathbb{Z}}}
\newcommand{\N}{\ensuremath{\mathbb{N}}}
\newcommand{\Omeg}{(0,1)^m}
\newcommand{\Wn}{W^{1,p}_0(\Omeg, \R^N)}

\newcommand{\I}[1]{\int\limits_{\Omeg} #1 \, dx}
\newcommand{\Iz}[1]{\int\limits_{(0,1)^2} #1 \, dx}
\newcommand{\supp}{\operatorname{supp}}

\newtheorem{definition}{Definition}[section]
\newtheorem{lemma}[definition]{Lemma}
\newtheorem{proposition}[definition]{Proposition}
\newtheorem{satz}[definition]{Theorem}
\newtheorem*{satz1}{Theorem}

\newtheorem{bemerkung}[definition]{Remark}
\newtheorem{korollar}[definition]{Corollary}

\theoremstyle{definition}

\numberwithin{equation}{section}
\numberwithin{definition}{section}

\DeclareCaptionType{mytype}[Figure][]

\begin{document}
\title{On the Approximation of Anisotropic Energy Functionals by Riemannian Energies via Homogenization} 
\author{Till Knoke} 

\maketitle 

\section*{Abstract}

In \cite{paper}, Braides, Buttazzo and Fragala proved the density of Riemannian energies in the class of Finsler energy functionals with respect to $\Gamma$-convergence in the one-dimensional case. In this thesis we prove that one of the main tools in \cite{paper}, a homogenization theorem, can be extended to arbitrary dimension, however, the density result cannot be generalized to higher dimensions. In fact, we construct counterexamples that show: there are anisotropic energy functionals, such as Finsler energies, Cartan functionals and their dominance functionals that cannot be $\Gamma$-approximated by Riemannian energies.

\section{Introduction}\pagenumbering{arabic}

In \cite{paper}, Braides, Buttazzo and Fragala established the density of isotropic Riemannian energy functionals in a class of in general anisotropic Finsler energy functionals with respect to the topology induced by $\Gamma$-convergence. To be more precise, for every Finsler energy functional
\begin{align*}
 L(u)=\int\limits_I \varphi(u(x), Du(x)) \, dx
\end{align*}
defined on curves $u: I \subset \R \rightarrow \R^N$, where $\varphi(s,z)$ is lower semicontinuous in $s$, convex and $2$-homogeneous in $z$ and $m_1 \lvert z \rvert^2 \leq \varphi(s,z) \leq m_2 \lvert z \rvert^2$ for every $(s,z) \in \R^N \times \R^N$ for some constants $0<m_1 \leq m_2$, there exists a sequence of Riemannian energies of the form
\begin{align*}
 L^n(u)=\int\limits_I a^n(u(x)) \lvert Du(x) \rvert^2 \, dx,
\end{align*}
where $a^n$ are lower semicontinuous functions bounded from above and below, such that the functionals $L^n$ $\Gamma$-converges to $L$ in the $(L^2(I, \R^N))$-topology. This means that for every sequence $u^n$ converging to $u$ in $L^2(I, \R^N)$ the \textit{liminf-inequality}
\begin{align*}
 L(u) \leq \liminf\limits_{n \rightarrow \infty} L^n(u^n)
\end{align*}
holds and that there exists a \textit{recovery sequence} $u^n$ converging to $u$ in $L^2(I, \R^N)$ satisfying the \textit{limsup-inequality}
\begin{align*}
 L(u) \geq \limsup\limits_{n \rightarrow \infty} L^n(u^n).
\end{align*}
The present work addresses the question if such a density result can be generalized to higher dimensions. We will discuss the Riemannian approximation of energy functionals of Finsler metrics (see \cite{Centore}, \cite{Fuglede}, \cite{Jost}, \cite{Mo}, \cite{Tachikawa})
\begin{align}
 L(u)=\int\limits_{\Omega} \varphi(u(x),Du(x)) \, dx \label{Finsler}
\end{align}
with a Lagrangian $\varphi: \R^N \times \R^{N \times m} \rightarrow \R$ $2$-homogeneous in the second variable, or of Cartan functionals (see \cite{PlatI}, \cite{PlatII}, \cite{vdM})
\begin{align}
 L(u)=\int\limits_{\Omega} \Phi(u(x),Du_1(x) \wedge Du_2(x)) \, dx \label{Cartanfunctional}
\end{align}
with a function $\Phi: \R^3 \times \R^3 \rightarrow \R$ positively $1$-homogeneous in the second variable, and of their dominance functionals (see \cite{Hildebrandt})
\begin{align}
 G(u)=\int\limits_{\Omega} g(u(x), Du(x)) \, dx \label{dominancefunctional}
\end{align}
with an integrand $g$ which is a dominance function of the parametric integrand $\Phi$ of the Cartan functional $L$. That is, the associated Lagrangian $f$ of $\Phi$, given by
\begin{align*}
 f(s,A):=\Phi(s,A_1 \wedge A_2) \text{ for any } s \in \R^3, A=(A_1, A_2) \in \R^3 \times \R^3
\end{align*}
satisfies $f(s,A) \leq g(s,A)$ with equality if and only if $\lvert A_1 \rvert^2=\lvert A_2 \rvert^2$ and $A_1 \cdot A_2 =0$. Since $\Gamma$-convergence implies the convergence of minimizers of the approximating Riemannian energies to a minimizer of the approximated functionals \eqref{Finsler}, \eqref{Cartanfunctional} or \eqref{dominancefunctional} under some mild assumptions (see \cite[Chapter 7]{maso}), one could hope to import some regularity results from the minimizers of the approximating functionals to the minimizers of the limit functionals. \\
The proof of the density result in \cite{paper} is based on a homogenization theorem (see \cite[Proposition 2.4]{paper}). Such a homogenization theorem holds also for uniformly almost periodic functions in the multi-dimensional case; see \cite[Theorem 15.3]{braidef}. As we will see in Chapter \ref{Homogentheo}, not every periodic function is uniformly almost periodic, but in Theorem \ref{th153} the homogenization result will be extended to all periodic functions in the following way (and this may be of independent interest):
\begin{satz1}
 Let $p>1$ and let $f: \R^m \times \R^N \times \R^{N \times m} \rightarrow \R$ be $[0,1]^m$-periodic in its first and $[0,1]^N$-periodic in its second variable satisfying
\begin{equation}
 c_1 \lvert A \rvert^p \leq f(x,s,A) \leq c_2 (1+ \lvert A \rvert^p) 
\end{equation}
for some $0<c_1 \leq c_2$ and all $(x,s,A) \in \R^m \times \R^N \times \R^{N \times m}$. Then there exists a quasi-convex function $f_{hom}: \R^{N \times m} \rightarrow \R$ such that for every bounded open subset $\Omega$ of $\R^m$ and every $u \in W^{1,p}(\Omega, \R^N)$ the $\Gamma$-limit (in the $L^p(\Omega, \R^N)$-topology)
\begin{equation}
 \Gamma(L^p(\Omega, \R^N))-\lim\limits_{\varepsilon \rightarrow 0} \int\limits_{\Omega} f\bigg(\frac{x}{\varepsilon},\frac{u(x)}{\varepsilon}, Du(x)\bigg) dx = \int\limits_{\Omega} f_{hom}(Du(x)) \, dx \nonumber
\end{equation}
exists, and the function $f_{hom}$ satisfies the equation
\begin{equation}
 f_{hom}(Y)=\lim\limits_{t \rightarrow \infty} \inf \left\{\frac{1}{t^m} \int\limits_{(0,t)^m} f(x, u(x)+Yx, Du(x)+Y)dx; u \in W^{1,p}_0((0,t)^m; \R^N)\right\} \nonumber
\end{equation}
for all $Y \in \R^{N \times m}$.
\end{satz1}
However, the approach used in \cite{paper} for the one-dimensional case can not be generalized to the multi-dimensional case, and we are going to show that we \textit{cannot} expect such a density result in higher dimensions. In Chapter \ref{Counterexamplessimp} we will see that an approximation of a Finsler metric is not always possible, at least if one of the Riemannian manifolds is a Euclidean domain. These results will be presented in Theorem \ref{bild} and Theorem \ref{urbild}. Furthermore, in Theorem \ref{cartanwachs} we prove that isotropic approximating sequences for Cartan functionals can not satisfy a certain growth condition if they exist at all. This growth condition would be expected intuitively since the Cartan functional itself satisfies this condition. Moreover, we will discuss the approximation of dominance functionals of Cartan functionals. These dominance functionals are interesting because every conformally parametrized minimizer $u$ (i.e. a function $u$ with $\lvert Du_1(x) \rvert^2=\lvert Du_2(x) \rvert^2$ and $Du_1(x) \cdot Du_2(x)=0$ for almost every $x \in \Omega$) of $L$ is a minimizer of $G$, too. This can easily be seen by the following inequality for every $v \in X$ (see \cite[Theorem 6.2]{vdM}):
\begin{align*}
 G(u) = L(u) \leq L(v)= \int\limits_{\Omega} f(v(x), Dv(x)) \, dx \leq \int\limits_{\Omega} g(v(x), Dv(x)) \, dx = G(v).
\end{align*}
However, in Theorem \ref{counteriso} we give a counterexample for the approximation of dominance functionals of non-even Cartan functionals by isotropic Riemannian energy functionals where the Riemannian manifold of the preimage of $u$ is a Euclidean domain $\Omega \subset \R^m$. Here, a Cartan functional is non-even, if there exists $(s,z) \in \R^3 \times \R^3$ so that $\Phi(s,z) \neq \Phi(s,-z)$ for the parametric integrand $\Phi$. In Theorem \ref{dom} this counterexample will be extended to certain dominance functionals of even Cartan functionals.
In Chapter \ref{Counterexamples} we show that we can drop the assumption that one of the Riemannian manifolds is a Euclidean domain under certain conditions on the approximating sequences, and we still find counterexamples for the approximation of all Finsler metrics (Theorem \ref{Gegenbeispielallg}), any Cartan functional (Theorem \ref{GegenbeispielCart}) and all perfect dominance functionals of Cartan functionals (Theorem \ref{GegenbeispielDom}). These conditions on the approximating sequences are used in Chapter \ref{PropApprox} to prove that an approximating sequence of an anisotropic energy with an integrand only depending on the values of the derivative $Du(x)$ can be chosen independently of $x$ and $u(x)$ as well (Theorem \ref{indepa} and Theorem \ref{indepb}) under these conditions. One of these conditions is quite technical providing a certain uniform absolute continuity of the respective recovery sequences of the approximating sequences. This condition is needed to prove an extension of \cite[Proposition 12.3]{braidef} from all open sets to all Borel sets (see Proposition \ref{measure}) in the following way:\\
\\
 Let $L^n$ be a sequence of energy functionals satisfying a growth condition and let every subsequence of $L^n$ satisfy the technical condition mentioned above. Then there exists a subsequence $L^{n_k}$ so that $F(u,E)=\Gamma(L^2(\Omega, \R^N))-\lim_{k \rightarrow \infty} L^{n_k}(u,E)$ exists for all $u \in W^{1,2}(\Omega,\R^N)$ and every Borel set $E \subset \Omega$ and $F(u, \cdot)$ is a Borel measure for every $u \in W^{1,2}(\Omega, \R^N)$.
\\
\\
This proposition is used to show that the approximating energy functionals of an anisotropic energy with an integrand only depending on the values of $Du(x)$ can be chosen independent of $u(x)$ so the technical assumption can be replaced by the assumption that the approximating sequences are independent of $u(x)$ in the first place.

\section{A Homogenization Theorem}\label{Homogentheo}

In \cite[Theorem 15.3]{braidef}, Braides and Defranceschi proved a homogenization result for uniformly almost periodic functions. In this section, we will see that not every periodic function is uniformly almost periodic and afterwards we will extend the homogenization theorem to all periodic functions. First we recall the definition for uniformly almost periodicity (see \cite[Definition 15.1]{braidef}).

\begin{definition}
 Let $(X, \lVert \cdot \rVert)$ be a complex Banach space. We say that a measurable function $v: \R^N \rightarrow X$ is uniformly almost periodic if it is the uniform limit of a sequence of trigonometric polynomials on $X$, i.e. $\lim\limits_{k \rightarrow \infty} \lVert P_k-v \rVert_{\infty}=0$ for some functions of the form $P_k(y)=\sum\limits_{j=1}^{r_k} x_j^k e^{i \lambda_j^k \cdot y}$ with $x^k_j \in X$, $\lambda_j^k \in \R^N$ and $r_k \in \N$. The definition easily extends to real Banach spaces.
\end{definition}

\begin{definition}
 A set $T \subset \R^N$ is relatively dense in $\R^N$ if there exists an inclusion length $L>0$ such that $T+[0,L)^N=\R^N$.
\end{definition}

By virtue of \cite[Theorem A.6]{braidef}, for an uniformly almost periodic function $f: \R^m \times \R^N \times \R^{N \times m} \rightarrow \R$ for all $Y \in \R^{N \times m}$ and $\eta>0$, the sets
\begin{align*}
 T_{\eta}^Y= \left\{ \tau \in \R^m; \lvert f(x +\tau, s+Y\tau, A)- f(x, s, A) \rvert < \eta (1+\lvert A \rvert^p)  \, \forall \, (x,s,A) \in \R^m \times \R^N \times \R^{N \times m} \right\}
\end{align*}
are relatively dense in $\R^N$.

Define
\begin{align*}
 f(x,s,A):=\begin{cases} 1 & s \in \Z^N \\
            2 & \text{otherwise}.
          \end{cases}
\end{align*}
Clearly, $f$ is $[0,1]^m$-periodic in its first variable and $[0,1]^N$-periodic in its second variable, but for $m=N=2$ and $Y=\begin{pmatrix} 1 & 1 \\ \sqrt{2} & \sqrt{2} \end{pmatrix}$ we have
\begin{align*}
 \lvert f(x +\tau, Y\tau, A)- f(x, 0, A) \rvert = \begin{cases} 0 & Y\tau \in \Z^2 \\
                                                   1 & \text{otherwise}.
                                                  \end{cases}
\end{align*}
Thus, $\tau \in T^Y_{\eta}$ is equivalent to $Y\tau \in \Z^2$ for all $\eta<1$. This implies that $T^Y_{\eta}=\left\{ \tau \in \R^2; \tau_1=-\tau_2 \right\}$. Now assume $T^Y_{\eta}$ were relatively dense in $\R^N$. Then there would be an inclusion length $L>0$. Let $x=\begin{pmatrix} 2L \\ 2L \end{pmatrix}$. If there were a $\tau \in T^Y_{\eta}$ and a $y \in [0,L)^2$ with $x=\tau+y$, we would deduce $L=2L-L \leq 2L-y_1 = x_1-y_1=\tau_1$ and $\tau_2=-\tau_1 \leq -L$. This would imply $y_2=x_2-\tau_2 \geq 2L+L=3L$, which is a contradiction to $y \in [0,L)^2$. Thus, $T^Y_{\eta}$ is not relatively dense in $\R^N$ and so, $f$ cannot be uniformly almost periodic. To extend \cite[Theorem 15.3]{braidef} to all functions which are $[0,1]^m$-periodic in its first variable and $[0,1]^N$-periodic in its second variable, we will follow the proof in \cite{braidef} and adjust it to our new setting. We start with a lemma similar to \cite[Proposition 15.4]{braidef}:

\begin{lemma}\label{prop154}
 Let $p>1$ and let $f: \R^m \times \R^N \times \R^{N \times m} \rightarrow \R$ be $[0,1]^m$-periodic in its first and $[0,1]^N$-periodic in its second variable satisfying $c_1 \lvert A \rvert^p \leq f(x,s,A) \leq c_2 (1+ \lvert A \rvert^p)$ for some $0<c_1 \leq c_2$ and all $(x,s,A) \in \R^m \times \R^N \times \R^{N \times m}$. Then for every sequence $(\varepsilon_j)$ of positive real numbers converging to $0$, there exists a subsequence $(\varepsilon_{j_k})$ and a quasi-convex function $\varphi: \R^{N \times m} \rightarrow \R$ such that for every bounded set $\Omega \subset \R^m$ the $\Gamma$-limit
\begin{equation}
 \Gamma(L^p(\Omega, \R^N))-\lim\limits_{k \rightarrow \infty} \int\limits_{\Omega} f\left(\frac{x}{\varepsilon_{j_k}},\frac{u(x)}{\varepsilon_{j_k}}, Du(x)\right) \, dx=\int\limits_{\Omega} \varphi (Du(x)) \, dx \nonumber
\end{equation}
exists for every $u \in W^{1,p}_0(\Omega, \R^N)$.
\end{lemma}

\begin{proof}
Let $\mathcal{A}(\Omega)$ denote the family of all open subsets of $\Omega$. By applying \cite[Proposition 12.3]{braidef} to the family of functionals $F_{\varepsilon}: W^{1,p}(\Omega, \R^N) \times \mathcal{A}(\Omega) \rightarrow [0, \infty]$ defined by $F_{\varepsilon}(u,U):=\int\limits_U f(\tfrac{x}{\varepsilon},\tfrac{u(x)}{\varepsilon}, Du(x)) \, dx$, we obtain the existence of a subsequence $(\varepsilon_{j_k})$ such that the limit
\begin{equation}
 F(u,U)=\Gamma(L^p(\Omega, \R^N))-\lim\limits_{k \rightarrow \infty} \int\limits_U f\left(\frac{x}{\varepsilon_{j_k}},\frac{u(x)}{\varepsilon_{j_k}}, Du(x)\right) \text{ } dx \label{Steffi13}
\end{equation}
exists for every $u \in W^{1,p}(\Omega, \R^N)$ and $U \in \mathcal{A}(\Omega)$, and the set function $F(u, \cdot)$ is the restriction of a Borel measure to $\mathcal{A}(\Omega)$. Obviously, $F(u,U)=F(v,U)$ whenever $U \in \mathcal{A}(\Omega)$ and $u=v$ almost everywhere on $U$. Let $u \in W^{1,p}(\Omega, \R^N)$ and $U \in \mathcal{A}(\Omega)$. Then $F(u,U) \leq \liminf\limits_{k \rightarrow \infty} F_{\varepsilon_{j_k}}(u,U) \leq c_2 \int\limits_U 1 + \lvert Du \rvert^p \text{ } dx$. Since the derivatives of the recovery sequence $(u_k)$ are equally bounded due to the growth condition of $f$ and by the weak compactness of reflexive Banach spaces, $(u_k)$ has a $W^{1,p}(\Omega,\R^N)$-weakly converging subsequence $u_{k_l}$. Then $(Du_{k_l})$ converges $L^p(\Omega, \R^{N \times m})$-weakly to a limit function $h \in L^p(\Omega, \R^{N \times m})$ and
\begin{align*}
\int\limits_{\Omega} u(x) D\varphi(x) \text{ } dx=\lim\limits_{l \rightarrow \infty} \int\limits_{\Omega} u_{k_l}(x) D\varphi(x) \text{ } dx &=-\lim\limits_{l \rightarrow \infty}  \int\limits_{\Omega} \varphi(x) Du_{k_l}(x) \text{ } dx
\end{align*}
which equals $-\int\limits_{\Omega} \varphi(x) h(x) \, dx$ for all $\varphi \in C_0^1(\Omega, \R^N)$, so $h$ is the weak derivative of $u$. Since every weakly convergent subsequence of $(Du_k)$ converges weakly to $Du$ and every subsequence has a weakly converging subsequence, the whole sequence converges weakly to $Du$ and so
\begin{equation}
 F(u,U) \geq \limsup\limits_{k \rightarrow \infty} F_{\varepsilon_{j_k}}(u_k,U) \geq c_1 \int\limits_U \lvert Du_k \rvert^p \text{ } dx \geq c_1 \int\limits_U \lvert Du \rvert^p \text{ } dx  \nonumber
\end{equation}
by the weak lower semicontinuity of norms. \\
Let $U \in \mathcal{A}(\Omega)$, $u \in W^{1,p}(\Omega, \R^N)$, $a \in \R^N$ and let $(u_k)$ be the recovery sequence in $W^{1,p}(\Omega, \R^N)$ such that $u_k \rightarrow u$ in $L^p(U, \R^N)$. 
Define the sequence $(a_k)$ in $\R^N$ such that $(a_k)_i:=\left\lfloor a_i/\varepsilon_{j_k} \right\rfloor \varepsilon_{j_k}$ for all $i \in \left\{ 1,2,..,N\right\}$. Thus, $(a_k)_i \leq a_i$ and $(a_k)_i \geq a_i - \varepsilon_{j_k}$, so $a_k$ converges to $a$ and $a_k/\varepsilon_{j_k} \in \Z^N$. Then by the periodicity we achieve
\begin{eqnarray}
 F(u+a,U) &\leq& \liminf\limits_{k \rightarrow \infty} \int\limits_U f\left( \frac{x}{\varepsilon_{j_k}}, \frac{u_k(x)}{\varepsilon_{j_k}} + \frac{a_k}{\varepsilon_{j_k}} , Du_k(x)\right) \text{ }dx \nonumber\\
&\leq& \limsup\limits_{k \rightarrow \infty} \int\limits_U f\left( \frac{x}{\varepsilon_{j_k}}, \frac{u_k(x)}{\varepsilon_{j_k}}, Du_k(x)\right) \text{ }dx = F(u,U). \nonumber
\end{eqnarray}
By a symmetry argument we get $F(u+a,U)=F(u,U)$. By these properties and the lower semicontinuity of the $\Gamma$-limit, we can apply \cite[Theorem 9.1]{braidef} to obtain the existence of a Carath�odory function $\varphi: \R^m \times \R^{N \times m} \rightarrow \R$ such that $F(u, \Omega)=\int\limits_{\Omega} \varphi (x, Du(x))\text{ }dx$ for every $u \in W^{1,p}(\Omega, \R^N)$.
Now fix $y, z \in \R^m$, $\rho>0$ and $Y \in \R^{N \times m}$, let $B(y, \rho)$ denote the ball with center $y$ and radius $\rho$, and let $(u_k)$ be a sequence in $W^{1,p}_0(B(y, \rho), \R^N)$ such that $u_k \rightarrow 0$ in $L^p(B(y, \rho), \R^N)$ and
\begin{equation}
 \lim\limits_{k \rightarrow \infty} F_{\varepsilon_{j_k}}(Yx+u_k, B(y,\rho))=F(Yx, B(y, \rho)) \label{Steffi11}
\end{equation}
and extend $u_k$ to $\R^m$ by $0$ outside of $B(y, \rho)$. Let $\tau_k$ and $\sigma_k$ be sequences in $\R^m$ defined by $(\tau_k)_i := \left\lfloor (z-y)_i/\varepsilon_{j_k} \right\rfloor \varepsilon_{j_k}$ and $(\sigma_k)_i := \left\lfloor (-Y \tau_k)_i/\varepsilon_{j_k} \right\rfloor \varepsilon_{j_k} + (Y \tau_k)_i$. Then $\tau_k \rightarrow z-y$, $\sigma_k \rightarrow 0$ and
\begin{align}
 \frac{\tau_k}{\varepsilon_{j_k}} \in \Z^m, \quad \frac{\sigma_k-Y\tau_k}{\varepsilon_{j_k}} \in \Z^N. \label{Steffi14}
\end{align}
 Define $v_k(x):=u_k(x)-\sigma_k$ and $w_k(x):=v_k(x-\tau_k)$. Then we obtain for $r>1$ by first using \eqref{Steffi11}, then transforming the integral over $B(y, \rho)$ to $\tau_k+B(y, \rho)$ and using the periodicity of $f$ and at last splitting the ball $\tau_k + B(y, \rho)$ into $B(z, \rho r)$ and its complement for $k$ large enough, using the growth condition on $f$ and using \eqref{Steffi13} that
\begin{eqnarray}
 F(Yx, B(y, \rho)) &\geq& F(Yx, B(z, \rho)) -\lvert B(z, \rho r) \setminus ( B(z,\rho)) \rvert c_2 (1+  \lvert Y \rvert^p )\nonumber\\
&\longrightarrow& F(Yx, B(z, \rho)) \text{ for } r \rightarrow 1.\nonumber
\end{eqnarray}
The opposite inequality $F(Yx, B(y, \rho)) \leq F(Yx, B(z, \rho))$ is again obtained by a symmetry argument, so $F(Yx, B(y, \rho)) = F(Yx, B(z, \rho))$. By \cite[Proposition 9.2]{braidef}, $\varphi$ is quasi-convex and independent of its first variable.
\end{proof}

The next lemma is similar to \cite[Proposition 15.5]{braidef}.

\begin{lemma}\label{prop155}
Let $p>1$ and let $f: \R^m \times \R^N \times \R^{N \times m} \rightarrow \R$ be $[0,1]^m$-periodic in its first and $[0,1]^N$-periodic in its second variable satisfying
\begin{equation}
 c_1 \lvert A \rvert^p \leq f(x,s,A) \leq c_2 (1+ \lvert A \rvert^p)  \nonumber
\end{equation}
for some $0<c_1 \leq c_2$ and all $(x,s,A) \in \R^m \times \R^N \times \R^{N \times m}$. Then the limit
\begin{equation}
 f_{hom}(Y)=\lim\limits_{t \rightarrow \infty} \inf\Bigg\{ \frac{1}{t^m} \int\limits_{(0,t)^m} f\left(x,Yx+u(x),Y+Du(x)\right)\text{ }dx; u \in W^{1,p}_0((0,t)^m, \R^N) \Bigg\} \nonumber
\end{equation}
exists for every $Y \in \R^{N \times m}$.
\end{lemma}

\begin{proof}
Let the matrix $Y \in \R^{N \times m}$ be fixed. For every $t>0$, define
\begin{equation}
 g_t:=\inf \Big\{ \frac{1}{t^m} \int\limits_{(0,t)^m} f(x, u(x)+Yx, Du(x)+Y) \text{ }dx; u \in W^{1,p}_0((0,t)^m, \R^N) \Big\} \label{Steffi15}
\end{equation}
and $u_t \in W^{1,p}_0((0,t)^m, \R^N)$ such that
\begin{equation}
 \frac{1}{t^m} \int\limits_{(0,t)^m} f(x, u_t(x)+Yx, Du_t(x)+Y) \text{ }dx \leq g_t + \frac{1}{t}. \label{Steffi12}
\end{equation}
Fix $t>0$ and $s>t+4$ and define $I_s$ as the set of all $z \in \Z^m$ such that $0  \leq z_j \leq \left\lfloor s/(t+4) \right\rfloor - 1$ for all $j \in \left\{1,2,..,m \right\}$, so
\begin{align}
 \lvert I_s \rvert=\left\lfloor \frac{s}{t+4} \right\rfloor^m. \label{Steffi16}
\end{align}
  Then, for every $z \in I_s$, choose $\sigma_z \in (t+4) z + (1,2]^m \cap \Z^m$, $\lambda_z \in Y \sigma_z + [0,1)^m \cap \Z^m$ and define $\tau_z \in \Z^m$ by $(\tau_z)_i:=(\sigma_z)_i-1$ for all $i \in \left\{1,2,..,m\right\}$. By these definitions, for $z \neq z'$, the following inequalities hold for all $i \in \left\{1,2,..,m \right\}$: $\lvert \sigma_z - \sigma_{z'} \rvert \geq t+3$, $\lvert \tau_z - \tau_{z'} \rvert \geq t+3$, $\lvert \sigma_z - \tau_{z'} \rvert \geq t+2$, $(\sigma_z)_i > 1$, $(\tau_z)_i > 0$, $(\sigma_z)_i \leq (t+4) (s/(t+4) - 1 ) +2 = s-t-2$, $(\tau_z)_i \leq  s-t-3$ (see Figure \ref{Bild}). Thus, $\forall \, z \in I_s$ we have $\tau_z + [0, t+2)^m \subset (0,s)^m$ and $\sigma_z + [0,t)^m \subset (0,s)^m$ and the sets $\tau_z+ [0,t+2)^m$ are disjoint. Define $B_z:= \sigma_z+[0,t)^m$, $A_z:= \tau_z+[0,t+2)^m \setminus B_z$ and $Q:=(0,s)^m \setminus \bigcup\limits_{z \in I_s} \tau_z + [0,t+2)^m$. Then 
\begin{align}
\lvert Q \rvert = s^m - (t+2)^m \lvert I_s \rvert\overset{\eqref{Steffi16}}{=}s^m - (t+2)^m \left\lfloor \frac{s}{t+4} \right\rfloor^m. \label{Steffi17}
\end{align}
 Define $u_s \in W^{1,p}_0((0,t)^m, \R^N)$ by
\begin{equation}
 u_s(x):= \left\{
  \begin{array}{ll}
    u_t(x-\sigma_z)-Y\sigma_z+\lambda_z & \text{if }x \in B_z \\ 
    0 & \text{if }x \in Q \\
    \text{linearly extended} & \text{if }x \in A_z. \\
  \end{array}
\right. \nonumber
\end{equation}
Then $u_s(x)=-Y\sigma_z+\lambda_z \in [0,1)^m$ for $x \in \partial B_z$ and $\lvert x'-x \rvert \geq 1$ for $x' \in Q, x \in B_z$, so $\lvert Du_s(x)+Y \rvert \leq c(Y)$ for $x \in A_z$. Now we can estimate $g_s$ by using \eqref{Steffi15} and the fact that $u_2 \in W^{1,p}_0((0,s)^m, \R^N)$, splitting the integral over $(0,s)^m$ into integrals over $Q$ and the sets $A_z$ and $B_z$ for $z \in I_s$  and using the growth condition on $f$, transforming the integrals over the sets $B_z$ into integrals over $[0,t)^m$ and using the periodicity of $f$, then using \eqref{Steffi12}, \eqref{Steffi17} and \eqref{Steffi16}:
\allowdisplaybreaks \begin{eqnarray}
 g_s &\leq& (1 - ( \tfrac{t+2}{t+4} - \tfrac{t+2}{s} )^m ) c_2 (1+\lvert Y \rvert^p) + \tfrac{(t+2)^m-t^m}{(t+4)^m} c_2 (1+ c(Y)^p)  \nonumber\\ 
&& + ( \tfrac{t}{t+4} )^m ( g_t + \tfrac{1}{t} ) \nonumber
\end{eqnarray}
Taking the limit, first as $s\rightarrow \infty$, then as $t \rightarrow \infty$, we obtain
\begin{eqnarray}
 &&\limsup\limits_{s \rightarrow \infty} g_s\leq \liminf\limits_{t \rightarrow \infty} \limsup\limits_{s \rightarrow \infty} (1 - ( \tfrac{t+2}{t+4} - \tfrac{t+2}{s} )^m ) c_2 (1+\lvert Y \rvert^p) \nonumber\\
&& \hspace{2 cm}+ \tfrac{(t+2)^m-t^m}{(t+4)^m} c_2 (1+ c(Y)^p) + ( \tfrac{t}{t+4} )^m ( g_t + \tfrac{1}{t} ) \nonumber\\
&=& \liminf\limits_{t \rightarrow \infty}  (1 - ( \tfrac{t+2}{t+4} )^m ) c_2 (1+\lvert Y \rvert^p)+ \tfrac{(t+2)^m-t^m}{(t+4)^m} c_2 (1+ c(Y)^p) + ( \tfrac{t}{t+4} )^m ( g_t + \tfrac{1}{t} ) \nonumber
\end{eqnarray}
which equals $\liminf\limits_{t \rightarrow \infty} g_t$. Thus, the limit exists and the proof is complete.

\end{proof}

Now we can extend \cite[Theorem 15.3]{braidef} to all periodic functions.

\begin{satz}\label{th153}
 Let $p>1$ and let $f: \R^m \times \R^N \times \R^{N \times m} \rightarrow \R$ be $[0,1]^m$-periodic in its first and $[0,1]^N$-periodic in its second variable satisfying
\begin{equation}
 c_1 \lvert A \rvert^p \leq f(x,s,A) \leq c_2 (1+ \lvert A \rvert^p) 
\end{equation}
for some $0<c_1 \leq c_2$ and all $(x,s,A) \in \R^m \times \R^N \times \R^{N \times m}$. Then there exists a quasi-convex function $f_{hom}: \R^{N \times m} \rightarrow \R$ such that for every bounded open subset $\Omega$ of $\R^m$ and every $u \in W^{1,p}(\Omega, \R^N)$ the limit
\begin{equation}
 \Gamma(L^p(\Omega, \R^N))-\lim\limits_{\varepsilon \rightarrow 0} \int\limits_{\Omega} f\bigg(\frac{x}{\varepsilon},\frac{u(x)}{\varepsilon}, Du(x)\bigg) dx = \int\limits_{\Omega} f_{hom}(Du(x)) \, dx \nonumber
\end{equation}
exists, and the function $f_{hom}$ satisfies the equation
\begin{equation}
 f_{hom}(Y)=\lim\limits_{t \rightarrow \infty} \inf \Bigg\{\frac{1}{t^n} \int\limits_{(0,t)^n} f(x, u(x)+Yx, Du(x)+Y) \, dx; u \in W^{1,p}_0((0,t)^m, \R^N)\Bigg\} \nonumber
\end{equation}
for all $Y \in \R^{N \times m}$.
\end{satz}

\begin{proof}
 Let $\Omega$ be a bounded open subset of $\R^m$ and $(\varepsilon_j)$ a sequence of positive real numbers converging to $0$. By Lemma \ref{prop154} there exist a subsequence $(\varepsilon_{j_k})$ and a quasi-convex function $\varphi: \R^{N \times m} \rightarrow \R$ such that the $\Gamma$-limit
\begin{equation}
 \Gamma(L^p(\Omega, \R^N))-\lim\limits_{k \rightarrow \infty} \int\limits_{\Omega} f \bigg(\frac{x}{\varepsilon_{j_k}}, \frac{u(x)}{\varepsilon_{j_k}}, Du(x) \bigg) \, dx=\int\limits_{\Omega} \varphi(Du(x)) \, dx \nonumber
\end{equation}
exists for every $u \in W^{1,p}(\Omega, \R^N)$. 
Now fix an arbitrary $Y \in \R^{N \times m}$, define $\Omega:=(0,1)^m$, $G_{\varepsilon_{j_k}}^{Yx}(u): L^p(\Omeg, \R^N) \rightarrow [0,\infty]$ by
\begin{equation}
G_{\varepsilon_{j_k}}^{Yx}(u):= \begin{cases}
    \displaystyle\int\limits_{\Omeg} f\bigg(\frac{x}{\varepsilon_{j_k}}, \frac{u(x)}{\varepsilon_{j_k}}, Du(x)\bigg) \, dx & \text{if }u-Yx \in W^{1,p}_0(\Omeg, \R^N) \\
    +\infty & \text{otherwise}                              
                                \end{cases}    \nonumber
\end{equation}
and $\psi(u):=\Gamma(L^p(\Omeg, \R^N))-\lim\limits_{k \rightarrow \infty} G_{\varepsilon_{j_k}}^{Yx}(u)$. The existence of this $\Gamma$-limit is granted by \cite[Proposition 11.7]{braidef}.\\
Let $T$ be the trace operator, let $u \in W^{1,p}(\Omeg, \R^N)$ so that $T(u-Yx) \neq 0$ and let $u_{\varepsilon_{j_k}}$ be a sequence in $W^{1,p}(\Omeg, \R^N)$ that converges to $u$ in $L^p(\Omeg, \R^N)$. Then there exists no subsequence $u_{\varepsilon_{j_{k_l}}} \in W^{1,p}_0(\Omeg, \R^N)+Yx$ because otherwise, by the continuity and the linearity of the trace operator,
\begin{eqnarray}
 0\!=\!\lim\limits_{l \rightarrow \infty} \left\Vert T(u_{\varepsilon_{j_{k_l}}}\!-\!u) \right\Vert_{L^p}\!=\!\lim\limits_{l \rightarrow \infty} \left\Vert T(u_{\varepsilon_{j_{k_l}}}\!-\!Yx)\!-\!T(u\!-\!Yx) \right\Vert_{L^p} \!=\!\left\Vert T(u\!-\!Yx) \right\Vert_{L^p}\!>\!0 \nonumber
\end{eqnarray}
would hold. That implies that
\begin{equation}
 \psi(u)\geq \limsup\limits_{k \rightarrow \infty} G_{\varepsilon_{j_k}}^{Yx}(u_{\varepsilon_{j_k}})=\infty \label{eins}
\end{equation}
for a recovery sequence $u_{\varepsilon_{j_k}}$. Now let $u \in W^{1,p}_0(\Omeg, \R^N)+Yx$. Then by \cite[Proposition 11.7]{braidef} 
\begin{align*}
\int\limits_{\Omeg} \varphi(Du(x)) dx &= \Gamma(L^p(\Omeg, \R^N))-\lim\limits_{k \rightarrow \infty} \int\limits_{\Omeg} f \bigg(\frac{x}{\varepsilon_{j_k}},\frac{u(x)}{\varepsilon_{j_k}}, Du(x)\bigg) dx 
\end{align*}
which equals the $\Gamma-\lim G_{\varepsilon_{j_k}}^{Yx}(u)=\psi(u)$. By that and ($\ref{eins}$) we get
\begin{eqnarray}
 &&\min\left\{\psi(u); u \in L^p(\Omeg, \R^N) \right\}=\min\left\{\psi(u); u \in \Wn+Yx \right\} \nonumber\\
&=& \min\Bigg\{\I{\varphi(Du(x))}; u \in \Wn+Yx\Bigg\} \nonumber\\
&=&\min\Bigg\{\I{\varphi(Du(x)+Y)}; u \in \Wn\Bigg\}. \label{zwei}
\end{eqnarray}
On the other hand, by \cite[Theorem 7.2]{braidef}
\begin{eqnarray}
 &&\min\left\{\psi(u); u \in L^p(\Omeg, \R^N) \right\}=\lim\limits_{k \rightarrow \infty} \inf \left\{G_{\varepsilon_{j_k}}^{Yx}(u); u \in L^p(\Omeg, \R^N) \right\} \nonumber\\
&=&\lim\limits_{k \rightarrow \infty} \inf \left\{G_{\varepsilon_{j_k}}^{Yx}(u); u \in \Wn+Yx \right\} \nonumber\\
&=&\lim\limits_{k \rightarrow \infty} \inf \Bigg\{\I{f \left(\frac{x}{\varepsilon_{j_k}}, \frac{u(x)}{\varepsilon_{j_k}}, Du(x)\right)}; u \in \Wn+Yx \Bigg\} \nonumber\\
&=&\lim\limits_{k \rightarrow \infty} \inf \Bigg\{\I{f \left(\frac{x}{\varepsilon_{j_k}}, \frac{u(x)+Yx}{\varepsilon_{j_k}}, Du(x)+Y\right)}; u \in \Wn \Bigg\}. \nonumber
\end{eqnarray}
By that and ($\ref{zwei}$) we get
\begin{align}
 &\min\left\{\I{\varphi(Du(x)+Y)}; u \in \Wn\right\} \nonumber\\
=&\lim\limits_{k \rightarrow \infty} \inf \left\{\I{f \left(\frac{x}{\varepsilon_{j_k}}, \frac{u(x)+Yx}{\varepsilon_{j_k}}, Du(x)+Y\right)}; u \in \Wn \right\}. \label{drei}
\end{align}
By the quasi-convexity of $\varphi$ and \cite[Remark 5.15]{braidef} $\varphi$ is $W^{1,p}$-quasi-convex and therefore the left-hand side equals $\varphi(Y)$, so, by substituting $y=x/\varepsilon_{j_k}$ in the integral on the right hand side and defining $T_k:=1/\varepsilon_{j_k}$ and $v(y):=T_k u(y/T_k)$, the equation can be written as
\begin{eqnarray}
\varphi(Y)=\lim\limits_{k \rightarrow \infty} \inf \Bigg\{ \frac{1}{T_k^m} \int\limits_{\left(0,T_k \right)^m} f\left(y,Yy+v(y),Y+Dv(y)\right) \text{ } dy; v \in W^{1,p}_0((0,T_k)^m, \R^N) \Bigg\}. \nonumber
\end{eqnarray}
By Lemma \ref{prop155} the limit
\begin{eqnarray}
&&f_{hom}(Y)=\lim\limits_{t \rightarrow \infty} \inf \left\{\frac{1}{t^m} \int\limits_{(0,t)^m} f(x, u(x)+Yx, Du(x)+Y)dx; u \in W^{1,p}_0((0,t)^m; \R^N)\right\} \nonumber \\
&=&\lim\limits_{k \rightarrow \infty} \inf \left\{ \frac{1}{T_k^m} \int\limits_{\left(0,T_k \right)^m} f\left(y,Yy+v(y),Y+Dv(y)\right) \text{ } dy; v \in W^{1,p}_0((0,T_k)^m, \R^N) \right\} \nonumber
\end{eqnarray}
exists for every $Y \in \R^{N \times m}$.

\end{proof}

In \cite{paper}, the Homogenization Theorem is used to prove the density of Riemannian metrics in the space of all Finsler metrics, which is not possible in higher dimensions, as we will see in the following sections.

\section{Counterexamples for the $\Gamma$-Density of Dirichlet Energies with Euclidean Domain or Target}\label{Counterexamplessimp}

We will start this section by defining some classes of metrics and functionals. Later, we will see that those classes can not be approximated by certain classes of Riemannian metrics. From now on, $\Omega$ will always denote a bounded open subset of $\R^m$ and $\mathcal{A}(\Omega)$ will denote the set of all open subsets of $\Omega$.

\begin{definition}
 For $0< c_1 \leq c_2$ we define $\mathcal{E}_{c_1 c_2}(\Omega)$ as the set of all energy functionals of Finsler metrics controlled from above and below respectively by $c_2$ and $c_1$ times the Euclidean norm $\lvert \cdot \rvert$ on $\Omega$, i.e. every $L \in \mathcal{E}_{c_1 c_2}(\Omega)$ can be written as $L(u)=\int\limits_{\Omega} \varphi(u(x),Du(x)) \, dx$ for every $u \in W^{1,2}(\Omega, \R^N)$ where $\varphi: \R^N \times \R^{N \times m} \rightarrow [0, \infty)$ satisfies the following conditions:
\begin{itemize}
 \item $s \mapsto \varphi(s,A) \text{ is lower semicontinuous for all } A \in \R^{N \times m}$,
 \item $A \mapsto \varphi(s,A) \text{ is convex and }2-\text{homogeneous for all } s \in \R^N$,
 \item $c_1 \lvert A \rvert^2 \leq \varphi(s,A) \leq c_2 \lvert A \rvert^2 \text{ for all }(s,A) \in \R^N \times \R^{N \times m}$.
\end{itemize}
Furthermore, define $\mathcal{E}(\Omega):= \bigcup\limits_{c_1, c_2 \geq 0} \mathcal{E}_{c_1 c_2}(\Omega)$.
\end{definition}

\begin{definition}
 We define $\mathcal{R}(\Omega)$ as the set of all energy functionals of Riemannian metrics in $\mathcal{E}(\Omega)$, i.e. $L \in \mathcal{R}(\Omega)$ if and only if $L \in \mathcal{E}(\Omega)$ and the integrand $\varphi$ can be expressed as $\varphi(s,A)=b_{ij}(s) A^i_{\alpha} A^j_{\alpha}$ for a coefficient matrix $(b_{ij})_{i,j \in \{1,..,N\}}$. The energy functional is called isotropic if the integrand $\varphi$ satisfies $\varphi(s,A)= b(s) \lvert A \rvert^2$, i.e. $b_{ij}(s)=b(s) \delta_{ij}$, and we denote the set of all isotropic energy functionals in $\mathcal{R}(\Omega)$ by $\mathcal{R}_I(\Omega)$.
\end{definition}

\begin{definition}
 Let $\Omega \subset \R^2$. We define $\mathcal{C}(\Omega)$ as the set of all Cartan functionals $L(u)=\int\limits_{\Omega} \Phi(u(x), Du_1(x) \wedge Du_2(x)) \, dx$ for all $u \in W^{1,2}(\Omega, \R^3)$ where the parametric integrand\\ $\Phi: \R^3 \times \R^{3} \rightarrow [0, \infty)$ satisfies the following conditions:
 \begin{itemize}
  \item $\Phi(s,tz)=t \Phi(s,z) \text{ for all } (s,z) \in \R^3 \times \R^3, t>0$,
  \item $m_1 \lvert z \rvert \leq \Phi(s,z) \leq m_2 \lvert z \rvert \text{ for some } 0<m_1 \leq m_2 \text{ and for all } (s,z) \in \R^3 \times \R^3$,
  \item $z \mapsto \Phi(s,z) \text{ is convex for all } s \in \R^3$.
 \end{itemize}
\end{definition}

\begin{definition}\label{assoclagr}
 Let $L(u)=\int\limits_{\Omega} \Phi(u(x), Du_1(x) \wedge Du_2(x)) \, dx \in \mathcal{C}(\Omega)$. We call $f(s,A):=\Phi(s,A_1 \wedge A_2)$ for $A=(A_1, A_2) \in \R^3 \times \R^3$ the associated Lagrangian of the parametric integrand $\Phi$.
\end{definition}

Now we will recall the definition of dominance functions (see e.g. \cite{Hildebrandt}).

\begin{definition}\label{domfunct}
 Let $\Phi(s, z)$ be the parametric integrand of a functional $L \in \mathcal{C}(\Omega)$ with the associated Lagrangian $f(s,A)$. Then a function $g: \R^3 \times \R^{3 \times 2}$ is said to be a dominance function for $\Phi$ if it is continuous and satisfies the following conditions:
\begin{itemize}
 \item $f(s,A)\leq g(s,A) \text{ for any } (s,A) \in \R^3 \times \R^{3 \times 2}$,
 \item $f(s,A)=g(s,A) \text{ if and only if } \lvert A_1 \rvert^2=\lvert A_2 \rvert^2 \text{ and } A_1 \cdot A_2=0$.
\end{itemize}
A dominance function $g \in C^0(\R^3 \times \R^{3 \times 2}) \cap C^2(\R^3 \times (\R^{3 \times 2} \setminus \{0\}))$ is called perfect if it satisfies the following conditions:
\begin{itemize}
 \item $g(s,tA)=t^2 g(s,A) \text{ for all } t>0, (s,A) \in \R^3 \times \R^{3 \times 2}$,
 \item $\exists \, \mu_1, \mu_2 \in \R, 0< \mu_1 \leq \mu_2 \text{ so that } \mu_1 \lvert A \rvert^2 \leq g(s,A) \leq \mu_2 \lvert A \rvert^2 \text{ for all } (s,A) \in \R^3 \times \R^{3 \times 2}$,
 \item $\text{for any } R_0>0 \text{ there is a constant } \lambda_g(R_0) \text{ such that } \xi^T g_{AA}(s,A) \xi \geq \lambda_g(R_0) \lvert \xi \rvert^2 \\
 \text{for all } \lvert s \rvert \leq R_0 \text{ and } A, \xi \in \R^{3 \times 2}, A \neq 0$.
\end{itemize}

\end{definition}

\begin{definition}
 Let $L(u)=\int\limits_{\Omega} \Phi(u(x), Du_1(x) \wedge Du_2(x)) \, dx \in \mathcal{C}(\Omega)$ and let $g$ be a dominance function for $\Phi$. We call $G(u):=\int\limits_{\Omega} g(u(x), Du(x)) \, dx$ a dominance functional for $L$. If $g$ is a perfect dominance function for $\Phi$, we call $G$ a perfect dominance functional of $L$.
\end{definition}

\begin{satz}\label{counteriso}
Not every functional $L \in \mathcal{E}(\Omega)$ can be approximated by elements of $\mathcal{R}_I(\Omega)$.
\end{satz}

\begin{proof}
 From now on $e_i$ will denote the $i$th unit vector. Let $m=2, N=3, \Omega=(0,1)^2$ and $L$ be the energy functional of a function $\varphi$ satisfying $\varphi(s, (e_1 | e_2))> \varphi(s, (e_2 | e_1))$ for all $s \in \R^3$. Suppose there exists a sequence of Riemannian coefficients $b_n$ such that $\int\limits_{(0,1)^2} b_n(u(x)) \lvert Du(x) \rvert^2 \, dx$ $\Gamma(L^2(\Omega, \R^N))$-converges to $L$. Define $u_1(x):=(x_1, x_2, 0)$ and $u_2(x):=(x_2, x_1, 0)$. Let $u_2^n$ be the recovery sequence for $u_2$ and $u_1^n(x):=u_2^n(x_2, x_1)$. Then by substituting $(x_2, x_1)$ by $x$ and later resubstituting we achieve
\allowdisplaybreaks\begin{eqnarray}
 &&\limsup\limits_{n \rightarrow \infty} \Iz{b_n(u_2^n(x)) \lvert Du_2^n(x) \rvert^2} \leq \Iz{\varphi(u_2(x), Du_2(x))} \nonumber\\
 &<& \Iz{\varphi(u_1(x), Du_1(x))} \leq \liminf\limits_{n \rightarrow \infty} \Iz{b_n(u_1^n(x)) \lvert Du_1^n(x) \rvert^2} \nonumber\\
 &=& \liminf\limits_{n \rightarrow \infty} \Iz{b_n(u_2^n(x)) \left\lvert Du_2^n(x) \left(\begin{matrix} 0 & 1 \\ 1& 0 \end{matrix}\right) \right\rvert^2} = \liminf\limits_{n \rightarrow \infty} \Iz{b_n(u_2^n(x)) \lvert Du_2^n(x)\rvert^2} \nonumber
\end{eqnarray}
which is a contradiction, so there exists no such sequence.
\end{proof}
Such a function $\varphi$ can clearly be a perfect dominance function for a parametric integrand $\Phi$ of a Cartan functional in $\mathcal{C}(\Omega)$, at least if $\Phi(s,z) = \Phi(s, -z)$ is not true for every $(s,z) \in \R^3 \times \R^3$. This motivates the following definition.
\begin{definition}
 A Cartan functional is called even if for the parametric integrand $\Phi$ the equation $\Phi(s,z)=\Phi(s,-z)$ holds for every $(s,z) \in \R^3 \times \R^3$.
\end{definition}
With this definition and Theorem \ref{counteriso} not every dominance functional of a non-even Cartan functional in $\mathcal{C}(\Omega)$ can be approximated by elements of $\mathcal{R}_I(\Omega)$.
To see that not every dominance functional of an even Cartan functional in $\mathcal{C}(\Omega)$ can be approximated by elements of $\mathcal{R}_I(\Omega)$, we will need the following lemma.

\begin{lemma}\label{domodd}
 Let $g$ be a perfect dominance function of an even Cartan functional in $\mathcal{C}(\Omega)$. Then there is a perfect dominance function $\Tilde g$ which satisfies 
\begin{align*}
 \int\limits_{(0,1)^2} \Tilde g\left(A_1 x_1 + A_2 x_2,A \right) \text{ }dx < \int\limits_{(0,1)^2} \Tilde g(A_1 x_1 + A_2 x_2, (A_2 | A_1 )) \text{ }dx
\end{align*}
for a matrix $A \in \R^{N \times m}$.
\end{lemma}

\begin{proof}
Let again $m=2$, $N=3$, let $\Phi$ be the parametric integrand of an even Cartan functional in $\mathcal{C}(\Omega)$ and let $g$ be a perfect dominance function of $\Phi$. Then either there is a matrix $A \in \R^{3 \times 2}$ such that
\begin{equation}
\int\limits_{(0,1)^2} g\left(A_1 x_1 + A_2 x_2,A \right) \text{ }dx < \int\limits_{(0,1)^2} g(A_1 x_1 + A_2 x_2,(A_2 | A_1)) \text{ }dx \nonumber
\end{equation}
or for every $A \in \R^{3 \times 2}$
\begin{equation}
\int\limits_{(0,1)^2} g\left(A_1 x_1 + A_2 x_2,A \right) \text{ } dx = \int\limits_{(0,1)^2} g(A_1 x_1 + A_2 x_2,(A_2 | A_1)) \text{ }dx \nonumber
\end{equation}
holds, since if the reverse strict inequality were true for a $A \in \R^{3 \times 2}$, then for $\Tilde A=(A_2 | A_1)$ we obtain by the same substitutions as in the proof of Theorem \ref{counteriso} that
\begin{eqnarray}
 \int\limits_{(0,1)^2} g(\Tilde A_1 x_1 + \Tilde A_2 x_2, \Tilde A ) \text{ }dx < \int\limits_{(0,1)^2} g(\Tilde A_1 x_1 + \Tilde A_2 x_2, ( \Tilde A_2 | \Tilde A_1 ) ) \text{ } dx \nonumber
\end{eqnarray}
which is a contradiction to the assumption that there is no $A$ satisfying this inequality. In the second case, we can modify $g$ in the following way: Define 
\begin{equation}
B:=\left\{ \left( \begin{matrix} r \sin \theta \cos \varphi \\  r \sin \theta \sin \varphi \\ r \cos \theta \end{matrix} \right); r\geq 0, \varphi \in \left[-\frac{3\pi}{16},\frac{3\pi}{16}\right], \theta \in \left[\frac{\pi}{32}, \frac{7\pi}{32}\right] \right\} \nonumber 
\end{equation}
and $D:=[0,\infty) \times [\frac{\pi}{32},\frac{7\pi}{32}] \times [-\frac{3\pi}{16},\frac{3\pi}{16}]$. Further, define $F: D \rightarrow B$ by
\begin{equation}
 F(r,\theta,\varphi):=(r \sin \theta \cos \varphi, r \sin \theta \sin \varphi, r \cos \theta )^T \nonumber
\end{equation}
and $g_1$, $g_2: [0, \infty) \times [\frac{\pi}{16},\frac{3\pi}{16}] \times [-\frac{\pi}{5},\frac{\pi}{5}] \rightarrow \R$ by
\begin{equation}
g_1(r,\theta,\varphi):= \left\{
  \begin{array}{ll}
    0 & \text{if }(\theta,\varphi) \in \partial[\frac{\pi}{16},\frac{3\pi}{16}] \times \partial[-\frac{\pi}{5},\frac{\pi}{5}] \text{ or } r=0 \\ 
    r & \text{if }r>0 \text{ and }(\theta,\varphi)=(\frac{\pi}{6},\frac{\pi}{6}) \\
    \text{linearly extended} & \text{otherwise}, \\
  \end{array}
\right. \nonumber\\
\end{equation}

\begin{equation}
g_2(r,\theta,\varphi):= \left\{
  \begin{array}{ll}
    0 & \text{if }(\theta,\varphi) \in \partial[\frac{\pi}{16},\frac{3\pi}{16}] \times \partial[-\frac{\pi}{5},\frac{\pi}{5}] \text{ or } r=0 \\ 
    r & \text{if }r>0 \text{ and }(\theta,\varphi)=(\frac{\pi}{6},0) \\
    \text{linearly extended} & \text{otherwise}. \\
  \end{array}
\right. \nonumber\\
\end{equation}
Note that $F$ is a $C^2$-diffeomorphism for $r>0$. Then mollify $g_1, g_2$ with a mollifier $\eta_{\varepsilon}$. Choose $\varepsilon$ small enough, so that $\supp(\eta_{\varepsilon} \ast g_1) \subset\subset D$, $\supp(\eta_{\varepsilon} \ast g_2) \subset\subset D$ and
\begin{align}
(\eta_{\varepsilon} \ast g_1) \left(\frac{1}{\sqrt{2}}, \frac{\pi}{6},\frac{\pi}{6}\right)>\eta_{\varepsilon} \ast g_2 \left(\frac{1}{\sqrt{2}}, \frac{\pi}{6},\frac{\pi}{6}\right), (\eta_{\varepsilon} \ast g_2) \left(\frac{1}{\sqrt{2}}, \frac{\pi}{6},0\right)>\eta_{\varepsilon} \ast g_1 \left(\frac{1}{\sqrt{2}}, \frac{\pi}{6},0\right). \label{glattung}
\end{align}
Note that $F(\tfrac{1}{\sqrt{2}}, \tfrac{\pi}{6},\tfrac{\pi}{6})=\tfrac{1}{\sqrt{2}} (\frac{\sqrt{3}}{4}, \frac{1}{4}, \frac{\sqrt{3}}{2} )^T$ and $F(\tfrac{1}{\sqrt{2}}, \tfrac{\pi}{6},0)=\tfrac{1}{\sqrt{2}} (\frac{1}{2}, 0, \frac{\sqrt{3}}{2})^T$. Define $h_1, h_2:  B \rightarrow \R$ by $h_1(p):=\lvert p \rvert^k (\eta_{\varepsilon} \ast g_1) \circ F^{-1}(p)$ and $h_2(p):=\lvert p \rvert^k (\eta_{\varepsilon} \ast g_2) \circ F^{-1}(p)$. Here choose $k>2$, so that $h_1, h_2 \in C^2_0 (B)$. Then define $H: B \times B \rightarrow \R$ by
\begin{equation}
 H(A):=\lvert A \rvert^2 ( h_1(\tfrac{A_1}{\lvert A \rvert}) h_2(\tfrac{A_2}{\lvert A \rvert})). \nonumber
\end{equation}
By that definition, $H \in C^2_0((B \times B)\setminus \left\{0 \right\})$ and for $D=\left(\begin{matrix} \frac{1}{2} & \frac{\sqrt{3}}{4} \\[1.1ex]  0  & \frac{1}{4} \\[1.1ex] \frac{\sqrt{3}}{2} & \frac{\sqrt{3}}{2} \end{matrix} \right)$ and $\Tilde D$ being the matrix $D$ with interchanged columns we achieve
\begin{eqnarray}
 &&H(\Tilde D ) = (\tfrac{1}{\sqrt{2}})^k (\eta_{\varepsilon} \ast g_1) (\tfrac{1}{\sqrt{2}}, \tfrac{\pi}{6},\tfrac{\pi}{6}) (\tfrac{1}{\sqrt{2}})^k (\eta_{\varepsilon} \ast g_2) (\tfrac{1}{\sqrt{2}}, \tfrac{\pi}{6},0) \nonumber\\
&\overset{\eqref{glattung}}{>}& (\tfrac{1}{\sqrt{2}})^k (\eta_{\varepsilon} \ast g_2) (\tfrac{1}{\sqrt{2}}, \tfrac{\pi}{6},\tfrac{\pi}{6}) (\tfrac{1}{\sqrt{2}})^k (\eta_{\varepsilon} \ast g_1) (\tfrac{1}{\sqrt{2}}, \tfrac{\pi}{6},0) = H(D). \nonumber
\end{eqnarray}
Obviously, $H$ is quadratic and $0\lvert A \rvert^2 \leq H(A) \leq \lvert A \rvert^2$. Furthermore, since $H \in C^2(\R^3)$, there exists a constant $\lambda_H$ such that
\begin{equation}
 \xi H_{AA}(A) \xi \geq \lambda_H \lvert \xi \rvert^2 \text{ for } A, \xi \in \R^6, A \neq 0. \nonumber
\end{equation}
Thus, since there are no perpendicular vectors in $B$ and obviously $H \geq 0$ everywhere, $\Tilde g:=g+aH$ is still a perfect dominance function of $F$, if we choose $a>0$ small enough, so that $\lambda_g+a\lambda_H>0$, but for $D$ and $\Tilde D$ as above, we have
\begin{eqnarray}
 &&\int\limits_{(0,1)^2} \Tilde g(D_1 x_1 + D_2 x_2, ( D_2 | D_1 ) ) \text{ } dx = \int\limits_{(0,1)^2} g(D_1 x_1 + D_2 x_2, ( D_2 | D_1 ) ) +aH( \Tilde D ) \text{ } dx \nonumber\\
&>& \int\limits_{(0,1)^2} g(D_1 x_1 + D_2 x_2, D )+aH(D ) \text{ } dx =\int\limits_{(0,1)^2} \Tilde g(D_1 x_1 + D_2 x_2,D ) \text{ } dx. \nonumber
\end{eqnarray}
So if there is a perfect dominance function for an even $\Phi$, there always is a perfect dominance function $\Tilde g$ with
\begin{equation}
 \int\limits_{(0,1)^2} \Tilde g(A_1 x_1 + A_2 x_2,A ) \text{ } dx < \int\limits_{(0,1)^2} \Tilde g(A_1 x_1 + A_2 x_2, ( A_2 | A_1 ) ) \text{ } dx \nonumber
\end{equation}
for some $A \in \R^{3 \times 2}$.
\end{proof}

\begin{satz}\label{dom}
Not every perfect dominance functional of an even Cartan functional in $\mathcal{C}(\Omega)$ can be approximated by elements of $\mathcal{R}_I(\Omega)$.
\end{satz}
 
\begin{proof}
Let $g$ be a perfect dominance function of an even Cartan functional in $\mathcal{C}(\Omega)$. By Lemma \ref{domodd} we get a perfect dominance function $\Tilde g$ which satisfies 
\begin{align*}
 \int\limits_{(0,1)^2} \Tilde g(A_1 x_1 + A_2 x_2,A ) \text{ } dx < \int\limits_{(0,1)^2} \Tilde g(A_1 x_1 + A_2 x_2, ( A_2 | A_1 ) ) \text{ } dx
\end{align*}
for some $A \in \R^{N \times m}$. Now suppose $\Tilde{g}$ could be approximated by elements of $\mathcal{R}_I(\Omega)$ with coefficients $b_n$. Then let $u_1(x):=A_2 x_1 + A_1 x_2$ and $u_2(x):=A_1 x_1 + A_2 x_2$. Let $u_2^n$ be the recovery sequence for $u_2$ and let $u_1^n(x):=u_2^n(x_2, x_1)$. Then again by substituting $(x_2, x_1)$ by $x$ and later resubstituting we achieve
\begin{eqnarray}
 &&\limsup\limits_{n \rightarrow \infty} \int\limits_{(0,1)^2} b_n(u_2^n(x)) \lvert Du_2^n(x) \rvert^2 \text{ } dx \leq \int\limits_{(0,1)^2} \Tilde g(A_1 x_1 + A_2 x_2, A) \text{ } dx \nonumber\\
&<& \int\limits_{(0,1)^2} \Tilde g(A_1 x_1 + A_2 x_2, ( A_2 | A_1 ) ) \text{ } dx = \int\limits_{(0,1)^2} \Tilde g(u_1(x), Du_1(x)) \text{ } dx \nonumber\\
&\leq& \liminf\limits_{n \rightarrow \infty} \int\limits_{(0,1)^2} b_n(u_2^n(x_2,x_1)) \left\lvert Du_2^n(x_2, x_1) \left(\begin{matrix} 0 & 1 \\ 1 & 0 \end{matrix}\right) \right\rvert^2 \text{ } dx = \liminf\limits_{n \rightarrow \infty} \int\limits_{(0,1)^2} b_n(u_2^n(x)) \lvert Du_2^n(x) \rvert^2 \text{ } dx, \nonumber
\end{eqnarray}
which is a contradiction, so not every perfect dominance functional of an even Cartan functional in $\mathcal{C}(\Omega)$ can be approximated by elements of $\mathcal{R}_I(\Omega)$.
 \end{proof}

\begin{satz}\label{cartanwachs}
Let $L: W^{1,2}(\Omega, \R^N) \rightarrow [0, \infty) \in \mathcal{C}(\Omega)$ with parametric integrand $\Phi: \R^3 \times \R^3 \rightarrow [0, \infty)$. If $L$ can be approximated by a sequence of elements of $\mathcal{R}_I(\Omega)$ with coefficients $b_n: \R^N \rightarrow [0, \infty)$, then there exists no $c_1>0$ so that
\begin{equation}
 c_1 \lvert z \rvert^2  \leq b_n(s) \lvert z \rvert^2 \label{Widerspruch}
\end{equation}
for all $(s,z) \in \R^3 \times \R^3$ and all $n \in \N$.
\end{satz}

\begin{proof}
Define $u(x):=(x_1+x_2,x_1+x_2,x_1+x_2)$ and let $u^n$ be the recovery sequence for $u$. Suppose there is $c_1>0$ satisfying (\ref{Widerspruch}). Then
\begin{eqnarray}
 &&0= \int\limits_{\Omega} \Phi(u(x), Du_1(x) \wedge Du_2(x)) \text{ } dx \nonumber\\
 &\geq& \limsup\limits_{n \rightarrow \infty} \int\limits_{\Omega} b_n(u^n(x)) \lvert Du^n(x) \rvert^2 \text{ } dx \geq \limsup\limits_{n \rightarrow \infty} c_1 \int\limits_{\Omega} \lvert Du^n(x) \rvert^2 \text{ } dx. \label{schwach}
\end{eqnarray}
Thus, $Du^n$ is a bounded sequence in $L^2(\Omega, \R^{N \times m})$, so because of the weak compactness of reflexive Banach spaces there exists a $L^2(\Omega, \R^{N \times m})$-weakly converging subsequence $u^{n_k}$. Let $h \in L^2(\Omega, \R^{N \times m})$ be the limit of this subsequence. Then as in the proof of Lemma \ref{prop154}, $h$ is the weak derivative of $u$. Since every weakly convergent subsequence of $(Du^n)$ converges weakly to $Du$, and since every subsequence has a weakly converging subsequence, the whole sequence converges weakly to $Du$ and so we have by the weak lower semicontinuity of norms
\begin{equation}
  \limsup\limits_{n \rightarrow \infty} c_1 \int\limits_{\Omega} \lvert Du^n(x) \rvert^2 \text{ } dx \geq c_1 \int\limits_{\Omega} \lvert Du(x) \rvert^2 \text{ } dx = 6 c_1>0. \nonumber
\end{equation}
Together with (\ref{schwach}), this is a contradiction, so there exists no such $c_1>0$.
\end{proof}

In Theorem \ref{counteriso}, we have seen that not every functional $L \in \mathcal{E}(\Omega)$ can be approximated by elements of $\mathcal{R}_I(\Omega)$. The next theorem shows that not every functional $L \in \mathcal{E}(\Omega)$ can be approximated by elements of $\mathcal{R}(\Omega)$, i.e. that the isotropy is not the reason for which the approximation does not always work.

\begin{satz}\label{bild}
Not every functional $L \in \mathcal{E}(\Omega)$ can be approximated by elements of $\mathcal{R}(\Omega)$.
\end{satz}

\begin{proof}
Let $m=2$, $N=3$, $\Omega=(0,1)^2$ and let $L \in \mathcal{E}(\Omega)$ with an integrand $\varphi$ satisfying $\varphi(s, (e_1 | e_2)) > \varphi(s, (e_2 | e_1) )$ for all $s \in \R^N$. Suppose there exists a sequence of coefficients $b^n$ such that $\int\limits_{(0,1)^2} b^n_{i j} (u(x)) u^i_{\alpha}(x) u^j_{\alpha}(x) \text{ } dx$ $\Gamma(L^2(\Omega, \R^N))$-converges to $L$. Define $u_1(x):=(x_1, x_2, 0)$ and $u_2(x):=(x_2, x_1, 0)$. Let $u_2^n$ be the recovery sequence for $u_2$ and $u_1^n(x):=u_2^n(x_2, x_1)$. Then
\begin{equation}
 Du_1^n (x)= \begin{pmatrix} (u_2^n)^1_2 & (u_2^n)^1_1 \\[1.1ex] (u_2^n)^2_2 & (u_2^n)^2_1 \\[1.1ex] (u_2^n)^3_2 & (u_2^n)^3_1 \end{pmatrix} \begin{pmatrix} x_2 \\ x_1 \end{pmatrix}. \nonumber
\end{equation}
Thus, by the same computations as in the proof of Theorem \ref{counteriso}
\begin{align*}
 &\limsup\limits_{n \rightarrow \infty} \Iz{b^n_{i j} (u_2^n(x)) (u_2^n)^i_{\alpha}(x) (u_2^n)^j_{\alpha}(x)} < \liminf\limits_{n \rightarrow \infty} \Iz{b^n_{i j} (u_1^n(x)) (u_1^n)^i_{\alpha}(x) (u_1^n)^j_{\alpha}(x)} \\
 =& \liminf\limits_{n \rightarrow \infty} \Iz{b^n_{i j} (u_2^n(x_2, x_1)) (u_2^n)^i_{\alpha}\begin{pmatrix} x_2 \\ x_1 \end{pmatrix} (u_2^n)^j_{\alpha}\begin{pmatrix} x_2 \\ x_1 \end{pmatrix}} = \liminf\limits_{n \rightarrow \infty} \Iz{b^n_{i j} (u_2^n(x)) (u_2^n)^i_{\alpha}(x) (u_2^n)^j_{\alpha}(x)} 
\end{align*}
holds, which is a contradiction, so there exists no such sequence. 
\end{proof}

In addition to the Riemannian metrics which are covered by $\mathcal{R}(\Omega)$, we might be interested in the behavior of sequences of energy functionals of maps $u: R_1 \rightarrow R_2$, where both $R_1, R_2$ are Riemannian manifolds, not only $R_2$. The energy functional of such a map $u$ will then, up to a constant factor, be defined by $L(u):=\int\limits_{\Omega} a^{\alpha \beta}(x) b_{ij}(u(x)) u^i_{\alpha}(x) u^j_{\beta}(x) \, dx$, where $a^{\alpha \beta}$ is the inverse of the metric tensor $a_{\alpha \beta}$ of $R_1$ and $b_{ij}$ is the metric tensor of $R_2$ (c.f. \cite[Chapter 8]{Jost}). At first, assume that $R_2$ is the Euclidean space $\R^N$.

\begin{definition}
 We define $\mathcal{I}(\Omega)$ as the set of all energy functionals of a map $u$ mapping from a Riemannian manifold to $\R^N$, i.e. $L \in \mathcal{I}(\Omega)$ if and only if $L(u)=\int\limits_{\Omega} a^{\alpha \beta}(x) u^i_{\alpha}(x) u^i_{\beta}(x) \, dx$ for a metric tensor $a_{\alpha \beta}$ of a Riemannian manifold.
\end{definition}

\begin{satz}\label{urbild}
Not every functional $L \in \mathcal{E}(\Omega)$ can be approximated by elements of $\mathcal{I}(\Omega)$.
\end{satz}

\begin{proof}
Let $m=2$, $N=3$, $\Omega=(0,1)^2$ and let $L \in \mathcal{E}(\Omega)$ with an integrand $\varphi$ satisfying $\varphi(s, (e_1 | e_2)) > \varphi(s, (e_2 | e_1))$ for all $s \in \R^N$. Suppose there exists a sequence of coefficients $a_n$ such that $\int\limits_{(0,1)^2} a_n^{\alpha \beta} (x) u^i_{\alpha}(x) u^i_{\beta}(x) \, dx$ $\Gamma(L^2(\Omega,$ $\R^N))$-converges to $L$. Define $u_1(x):=(x_1, x_2, 0)$ and $u_2(x):=(x_2, x_1, 0)$. Let $u_2^n$ be the recovery sequence for $u_2$ and $u_1^n(x):=(e_2 | e_1 | e_3) u_2^n(x)$. Then
\begin{equation}
 Du_1^n (x)= \begin{pmatrix} (u_2^n)^2_1 & (u_2^n)^2_2 \\[1.1ex] (u_2^n)^1_1 & (u_2^n)^1_2 \\[1.1ex] (u_2^n)^3_1 & (u_2^n)^3_2 \end{pmatrix}(x). \nonumber
\end{equation}
Thus, by computations analogous to those in the proof of Theorem \ref{counteriso} and then simply changing the order of summation
\allowdisplaybreaks\begin{eqnarray}
 &&\limsup\limits_{n \rightarrow \infty} \Iz{a_n^{\alpha \beta} (x) (u_2^n)^i_{\alpha}(x) (u_2^n)^i_{\beta}(x)} \leq \Iz{\varphi(u_2(x), Du_2(x))} \nonumber\\
 &<& \liminf\limits_{n \rightarrow \infty} \Iz{a_n^{\alpha \beta} (x) (u_1^n)^i_{\alpha}(x) (u_1^n)^i_{\beta}(x)} = \liminf\limits_{n \rightarrow \infty} \Iz{a_n^{\alpha \beta} (x) (u_2^n)^i_{\alpha}(x) (u_2^n)^i_{\beta}(x)} \nonumber
\end{eqnarray}
holds, which is a contradiction, so there exists no such sequence.
 
\end{proof}

We have now seen that an approximation of all functionals in $\mathcal{E}(\Omega)$ with Riemannian energy functionals is not possible if either $R_1$ or $R_2$ is a Euclidean domain. Therefore, in Chapter \ref{Counterexamples} we will discuss the behavior of sequences of energy functionals of a map $u$ mapping from a Riemannian manifold $R_1$ to another Riemannian manifold $R_2$. In the next theorem we will see that the $\Gamma(L^2(\Omega, \R^N))$-limit of a sequence of elements of $\mathcal{R}(\Omega)$ with an oscillation in the coefficients $b_{ij}$ must be given by a function $F(u)=\int\limits_{\Omega} \varphi(Du(x)) \, dx$, where the function $\varphi$ is even with respect to permuting columns. This is a structural restriction for classes which could be approximated by such elements of $\mathcal{R}(\Omega)$.

\begin{satz}
Let $L$ be the $\Gamma(L^2(\Omega, \R^N))$-limit of a sequence of elements of $\mathcal{R}(\Omega)$ with integrands $\varphi^n$ defined by $\varphi^n(s,z):= b_{ij}(ns) z^i_{\alpha} z^j_{\alpha}$ with $[0,1]^N$-periodic, measurable and bounded coefficients $b_{ij}$. Then $L(u)=\int\limits_{\Omega} \varphi(Du(x)) \, dx$, where the function $\varphi$ is even with respect to permuting columns. 
\end{satz}

\begin{proof}
By Theorem \ref{th153}, $L(u)=\int\limits_{\Omega} \varphi(Du(x)) \, dx$ for a function $\varphi$, where $\varphi(A)$ is given by
\begin{align*}
 \lim\limits_{t \rightarrow \infty} \inf \Bigg\{\frac{1}{t^m} \int\limits_{(0,t)^m} b_{ij}\left(u(x)+Ax\right) \left(Du(x)+A\right)^i_{\alpha} \left(Du(x)+A\right)^j_{\alpha} \text{ } dx; u \in W^{1,p}_0((0,t)^m, \R^N)\Bigg\}.
\end{align*}
For every $t \in \R$, let $u^t_k$ be a minimizing sequence in $W^{1,p}_0((0,t)^m, \R^N)$ such that
\begin{equation}
 \varphi(A)=\lim\limits_{t \rightarrow \infty} \lim\limits_{k \rightarrow \infty} \frac{1}{t^m} \int\limits_{(0,t)^m} b_{ij}\left(u^t_k(x)+Ax\right) \left( Du^t_k(x)+A\right)^i_{\alpha} \left( Du^t_k(x)+A\right)^j_{\alpha} \text{ } dx. \nonumber
\end{equation}
Let $\Tilde A$ be the matrix $A$ with permuted columns $l_1, l_2$, let $\Tilde I$ be the identity matrix with permuted columns $l_1,l_2$, let $\Tilde x$ be the vector $x$ with permuted elements $l_1,l_2$ and let $\Tilde u^t_k$ be defined by $\Tilde u^t_k (x)= u^t_k (\Tilde x)$. Then obviously, $\Tilde A \Tilde x=Ax$. Furthermore, $D\Tilde u^t_k(x)=D\left[u^t_k(\Tilde x)\right]=Du^t_k(\Tilde x) \Tilde I$. This yields
\begin{align*}
 \varphi(\Tilde A)&\leq\lim\limits_{t \rightarrow \infty} \lim\limits_{k \rightarrow \infty} \frac{1}{t^m} \int\limits_{(0,t)^m} b_{ij}\left(\Tilde u^t_k(x)+\Tilde Ax\right) \left( D\Tilde u^t_k(x)+\Tilde A\right)^i_{\alpha} \left( D\Tilde u^t_k(x)+\Tilde A\right)^j_{\alpha} \text{ } dx \\
&=\lim\limits_{t \rightarrow \infty} \lim\limits_{k \rightarrow \infty} \frac{1}{t^m} \int\limits_{(0,t)^m} b_{ij}\left(u^t_k(\Tilde x)+\Tilde Ax\right) \left( Du^t_k(\Tilde x) \Tilde I+\Tilde A\right)^i_{\alpha} \left( Du^t_k(\Tilde x)\Tilde I+\Tilde A\right)^j_{\alpha} \text{ } dx \\
&=\lim\limits_{t \rightarrow \infty} \lim\limits_{k \rightarrow \infty} \frac{1}{t^m} \int\limits_{(0,t)^m} b_{ij}\left(u^t_k(x)+\Tilde A \Tilde x\right) \left( Du^t_k(x) +A\right)^i_{\alpha} \left( Du^t_k(x) +A\right)^j_{\alpha} \text{ } dx = \varphi(A). 
\end{align*}
By symmetry arguments, we get $\varphi(A)\leq \varphi(\Tilde A)$, so there must be equality and thus, $\varphi$ is even with respect to permuting columns.
\end{proof}

\section{Properties of Approximating Sequences}\label{PropApprox}

In Chapter \ref{Counterexamples}, we will see that not all energy functionals of Finsler metrics, Cartan functionals or perfect dominance functionals can be approximated by sequences of energy functionals of maps mapping from one Riemannian manifold $R_1$ to another Riemannian manifold $R_2$ defined by $L^n(u,B)=\int\limits_{B} a_n^{\alpha \beta}(x) b_{ij}^n (x,u(x)) u_{\alpha}^i (x) u_{\beta}^j(x) \, dx$ satisfying the following conditions:
\begin{align} \exists \, M \in \R: \sup\limits_{n,x,\alpha,\beta} \lvert a_n^{\alpha \beta} (x) \rvert \leq M, \quad \sup\limits_{n,x,s,i,j} \lvert b_{ij}^n (x,s) \rvert \leq M, \label{bounded}
\end{align}
\begin{align} \exists \, c_1 >0: a_n^{\alpha \beta}(x) b_{ij}^n(x,s) A_{\alpha}^i A_{\beta}^j \geq c_1 \lvert A \rvert^2 \quad \text{for all } (x,s,A) \in \R^m \times \R^N \times \R^{N \times m}, \label{growth} 
\end{align}
\begin{align} \exists \, x_0: \forall \, \varepsilon>0 \quad \exists \, \delta>0: \quad \lvert a_n^{\alpha \beta} (x) - a_n^{\alpha \beta} (x_0) \rvert < \varepsilon \quad \forall n \in \N, \alpha, \beta \in \left\{ 1,..,m \right\} \nonumber\\
\text{and for almost every } x \in B(x_0, \delta), \label{UC} 
       \end{align}
there is a bounded continuous function $\omega: \R^+ \rightarrow \R^+$ such that $\omega(0)=0$ and \begin{align}
\forall &x \in \R^m, s_1, s_2 \in \R^N, A \in \R^{N \times m}: \nonumber\\
  &\lvert a_n^{\alpha \beta}(x) b_{ij}^n(x,s_1) A^i_{\alpha} A^j_{\beta}- a_n^{\alpha \beta}(x) b_{ij}^n(x,s_2) A^i_{\alpha} A^j_{\beta}\rvert \leq \omega(\lvert s_1 - s_2 \rvert) (1 + \lvert A \rvert^2), \label{fusco}                                                                                                           
                                                                                                          \end{align}
for every $\varepsilon>0$,  every Borel set $E \subset \Omega$  and every $u \in W^{1,2}(\Omega, \R^N)$ there exists an open set $U \supset E$ and a sequence $u^n$ in $W^{1,2}(\Omega, \R^N)$ converging to $u$ in $L^2(\Omega, \R^N)$ satisfying 
\begin{align*}
\limsup\limits_{n \rightarrow \infty} L^n(u^n,E)=\min \Bigg\{ \limsup\limits_{n \rightarrow \infty} L^n(v^n,E); v^n \text{ in } W^{1,2}(\Omega, \R^N) \text { and } v^n \rightarrow u \text{ in } L^2(\Omega, \R^N) \Bigg\}
\end{align*}
 so that
\begin{align}
 L^n(u^n,E) \geq L^n(u^n,U)-\varepsilon \quad \forall \, n \in \N. \label{explrecseq}
\end{align}
In this section, we will see some properties of $L^n$ which will be crucial for the proofs of the theorems in Chapter \ref{Counterexamples}.

\begin{bemerkung}\label{bemexpl}
 By \cite[Proposition 7.6]{braidef} there exists a sequence $u^n$ in $W^{1,2}(\Omega, \R^N)$ converging to $u$ in $L^2(\Omega, \R^N)$ and satisfying 
\begin{align*}
\limsup\limits_{n \rightarrow \infty} L^n(u^n,E)=\min \left\{ \limsup\limits_{n \rightarrow \infty} L^n(v^n,E); v^n \text{ in } W^{1,2}(\Omega, \R^N) \text { and } v^n \rightarrow u \text{ in } L^2(\Omega, \R^N) \right\}, 
\end{align*}
which is needed in \eqref{explrecseq}.
\end{bemerkung}

\begin{bemerkung}\label{bembounded}
 From \eqref{bounded} we can deduce $L^n(u,B) \leq M^2 m^2 N^2 \lVert Du \rVert^2_{L^2(B)}$ for every $B \subset \Omega$ since
\begin{align*}
 L^n(u,B) &= \int\limits_{B} a_n^{\alpha \beta}(x) b_{ij}^n (x,u(x)) u_{\alpha}^i (x) u_{\beta}^j(x) \, dx \leq \int\limits_{B} M^2  \sum\limits_{\alpha, \beta=1}^m \sum\limits_{i,j=1}^N \lvert u_{\alpha}^i (x) \rvert \cdot \lvert u_{\beta}^j(x)  \rvert \, dx\\
&\leq M^2 \int\limits_{B} m^2 N^2 \sup\limits_{\substack{i \in \{1,...,N\} \\ \alpha \in \{1,...,m\}}} \lvert u_{\alpha}^i (x) \rvert^2 \, dx \leq M^2 m^2 N^2 \lVert Du \rVert^2_{L^2(B)} \, dx.
\end{align*}

\end{bemerkung}

\begin{bemerkung}
 The conditions \eqref{bounded} and \eqref{growth} are in some way coherent because the limit functional satisfies those conditions. The conditions \eqref{UC}, \eqref{fusco} and \eqref{explrecseq} are needed for technical reasons.
\end{bemerkung}

\begin{bemerkung}
 There are sequences of functionals satisfying the condition \eqref{explrecseq} as we will see in Corollary \ref{exfus}.
\end{bemerkung}

\begin{lemma}\label{constrecseq}
 Let $L^n$ be a sequence of functionals defined by $L^n(u):=\int\limits_{\Omega} a_n^{\alpha \beta} b^n_{i j} u^i_{\alpha}(x) u^j_{\beta}(x) \, dx$ for some coefficients $a_n^{\alpha \beta}$ and $b^n_{ij}$ satisfying the conditions \eqref{bounded} and \eqref{growth}. Then if $L^n$ $\Gamma(L^2(\Omega, \R^N))$-converges to some functional $L$ which satisfies the growth condition $L(u) \leq c_2 \lVert Du \rVert_{L^2(\Omega)}^2$ for some $c_2>0$, the constant sequence $u$ is a recovery sequence for every $u \in W^{1,2}(\Omega, \R^N)$.
\end{lemma}

\begin{proof}
 Let $u \in W^{1,2}(\Omega, \R^N)$ and let $u^n$ be the recovery sequence for $L^n$. Then we have
\begin{align*}
 \infty> c_2 \lVert Du \rVert_{L^2(\Omega)} \geq L(u) \geq \limsup\limits_{n \rightarrow \infty} L^n(u^n) \geq \limsup\limits_{n \rightarrow \infty} c_1 \lVert Du^n \rVert_{L^2(\Omega)}^2,
\end{align*}
so $Du^n$ is bounded in $L^2(\Omega, \R^{N \times m})$. Due to the weak compactness of reflexive Banach spaces, there exists a subsequence $Du^{n_k}$ weakly converging in $L^2(\Omega, \R^{N \times m})$ to some $h \in L^2(\Omega, \R^{N \times m})$ and as in the proof of Lemma \ref{prop154}, $h$ is the weak derivative of $u$. Since every weakly convergent subsequence of $Du^n$ converges weakly to $Du$ and since every subsequence has a weakly convergent subsequence, the whole sequence converges weakly to $Du$. $L^n(u^n)$ equals
\begin{align*}
 & L^n(u) + \int\limits_{\Omega} a_n^{\alpha \beta} b_{ij}^n (u)^i_{\alpha}(x) (u^n-u)^j_{\beta}(x) \, dx +\int\limits_{\Omega} a_n^{\alpha \beta} b_{ij}^n (u^n-u)^i_{\alpha}(x) (u)^j_{\beta}(x) \, dx + L^n(u^n-u) \\
\geq& L^n(u) + \int\limits_{\Omega} a_n^{\alpha \beta} b_{ij}^n (u)^i_{\alpha}(x) (u^n-u)^j_{\beta}(x) \, dx +\int\limits_{\Omega} a_n^{\alpha \beta} b_{ij}^n (u^n-u)^i_{\alpha}(x) (u)^j_{\beta}(x) \, dx 
\end{align*}
and thus, we can deduce that $L(u)$ is greater than or equal to
\begin{align*}
\limsup\limits_{n \rightarrow \infty} L^n(u) + \liminf\limits_{n \rightarrow \infty} \int\limits_{\Omega} a_n^{\alpha \beta} b_{ij}^n (u)^i_{\alpha}(x) (u^n-u)^j_{\beta}(x) \, dx  +\liminf\limits_{n \rightarrow \infty} \int\limits_{\Omega} a_n^{\alpha \beta} b_{ij}^n (u^n-u)^i_{\alpha}(x) (u)^j_{\beta}(x) \, dx,
\end{align*}
which equals $\limsup\limits_{n \rightarrow \infty} L^n(u)$, since $\lvert \liminf\limits_{n \rightarrow \infty} \int\limits_{\Omega} a_n^{\alpha \beta} b_{ij}^n (u)^i_{\alpha}(x) (u^n-u)^j_{\beta}(x) \, dx \rvert=0$ and $\lvert \liminf\limits_{n \rightarrow \infty} \int\limits_{\Omega} a_n^{\alpha \beta} b_{ij}^n (u^n-u)^i_{\alpha}(x) (u)^j_{\beta}(x) \, dx \rvert =0$. To see that, let $u^{n_k}$ be a subsequence of $u^n$ so that
\begin{align*}
 \liminf\limits_{n \rightarrow \infty} \int\limits_{\Omega} a_n^{\alpha \beta} b_{ij}^n (u)^i_{\alpha}(x) (u^n-u)^j_{\beta}(x) \, dx = \lim\limits_{k \rightarrow \infty} \int\limits_{\Omega} a_{n_k}^{\alpha \beta} b_{ij}^{n_k} (u)^i_{\alpha}(x) (u^{n_k}-u)^j_{\beta}(x) \, dx.
\end{align*}
Then the sequence $a_{n_k}^{\alpha \beta} b_{ij}^{n_k}$ is bounded by $M^2$ so there exists another subsequence $a_{n_{k_l}}^{\alpha \beta} b_{ij}^{n_{k_l}}$ converging to some $c_{ij}^{\alpha \beta}$ for all $i,j \in \{1,..,N\}$ and $\alpha, \beta \in \{1,..,m\}$. This implies
\begin{align*}
 &\lvert \liminf\limits_{n \rightarrow \infty} \int\limits_{\Omega} a_n^{\alpha \beta} b_{ij}^n (u)^i_{\alpha}(x) (u^n-u)^j_{\beta}(x) \, dx \rvert \\
 \leq& \lim\limits_{l \rightarrow \infty} \lvert a_{n_{k_l}}^{\alpha \beta} b_{ij}^{n_{k_l}} - c_{ij}^{\alpha \beta} \rvert \int\limits_{\Omega} \lvert (u)^i_{\alpha}(x) (u^{n_{k_l}}-u)^j_{\beta}(x) \rvert \, dx  + \lvert \int\limits_{\Omega} c_{ij}^{\alpha \beta} (u)^i_{\alpha}(x) (u^{n_{k_l}}-u)^j_{\beta}(x) \, dx \rvert = \, 0.
\end{align*}
The equality $\lvert \liminf\limits_{n \rightarrow \infty} \int\limits_{\Omega} a_n^{\alpha \beta} b_{ij}^n (u^n-u)^i_{\alpha}(x) (u)^j_{\beta}(x) \, dx \rvert =0$ can be proven in the same way, so the constant sequence $u^n=u$ is a recovery sequence for $u$.
\end{proof}

\begin{korollar}\label{exfus}
 Let $L^n$ be a sequence of functionals defined by $L^n(u):=\int\limits_{\Omega} a_n^{\alpha \beta} b^n_{i j} u^i_{\alpha}(x) u^j_{\beta}(x) \, dx$ for some coefficients $a_n^{\alpha \beta}$ and $b^n_{ij}$ satisfying the conditions \eqref{bounded} and \eqref{growth}. Then $L^n$ satisfies the condition \eqref{explrecseq} as well.
\end{korollar}

\begin{proof}
 Let $E \subset \Omega$ be a Borel set. By Lemma \ref{constrecseq} for every $u \in W^{1,2}(\Omega, \R^N)$ the constant sequence $u$ is a recovery sequence, i.e. 
\begin{align*}
\limsup\limits_{n \rightarrow \infty} L^n(u,E)=\min \left\{ \limsup\limits_{n \rightarrow \infty} L^n(v^n,E); v^n \text{ in } W^{1,2}(\Omega, \R^N) \text{ and } v^n \rightarrow u \text{ in } L^2(\Omega, \R^N) \right\}.
\end{align*}
Then by Remark \ref{bembounded}, for every $U \in \mathcal{A}(\Omega)$ with $E \subset U$ we achieve that $L^n(u,U) \leq L^n(u, E) + \int\limits_{U \setminus E} M^2 N^2 m^2 \lvert Du(x) \rvert^2 \, dx$. By the outer regularity of the Lebesgue measure, we have $\mathcal{L}(E)= \inf \left\{ \mathcal{L}(U); U \in \mathcal{A}(\Omega), E \subset U \right\}$. Let $U_l$ be a sequence in $\mathcal{A}(\Omega)$ such that $E \subset U_l$ for every $l \in \N$ and $\lim\limits_{l \rightarrow \infty} \mathcal{L}(U_l)=\mathcal{L}(E)$. Then we have
\begin{align*}
 \lim\limits_{l \rightarrow \infty} \int\limits_{U_l \setminus E} M^2 N^2 m^2 \lvert Du(x) \rvert^2 \, dx =\lim\limits_{l \rightarrow \infty} \int\limits_{\Omega} M^2 N^2 m^2 \lvert Du(x) \rvert^2 \cdot \mathbbm{1}_{U_l \setminus E}(x) \, dx =0
\end{align*}
by Lebesgue's Dominated Convergence Theorem (\cite[Theorem 1.8]{lieblos}). Let $\varepsilon>0$ and choose $L \in \N$ so that $\int\limits_{U_L \setminus E} M^2 N^2 m^2 \lvert Du(x) \rvert^2 \, dx < \varepsilon$. Then for this set $U_L$ we achieve $L^n(u,U_L) \leq L^n(u, E) + \varepsilon$ so the condition \eqref{explrecseq} is satisfied.
\end{proof}

\begin{lemma}\label{local}
 Let $L^n(u,\Omega)$ be a sequence of energy functionals $\Gamma(L^2(\Omega, \R^N))$-converging to the energy functional $L(u,\Omega)=\int\limits_{\Omega} \varphi(u(x), Du(x)) \, dx$ for a non-negative function $\varphi$ satisfying the growth condition $\varphi(s,A) \leq c_2 \lvert A \rvert^2$ for some $c_2>0$ and for all $s \in \R^N, A \in \R^{N \times m}$, and let $L^n$ satisfy the conditions \eqref{bounded}, \eqref{growth} and \eqref{fusco}. Then $L^n(u,U)$ $\Gamma(L^2(\Omega, \R^N))$-converges to $L(u,U)$ for every $U \in \mathcal{A}(\Omega)$ and $u \in W^{1,2}(\Omega, \R^N)$.
\end{lemma}

\begin{proof}
 By \cite[Theorem 2.4]{fusco} every subsequence $L^{n_k}$ of $L^n(u,U)$ has another subsequence $L^{n_{k_q}}$ which $\Gamma(L^2(\Omega, \R^N))$-converges to a function $F(u,U)=\int\limits_U g(x,u(x),Du(x)) \, dx$ for every $U \in \mathcal{A}(\Omega)$.
Now choose arbitrary $x \in \Omega$ and $\varepsilon>0$ such that $B(x, \varepsilon) \subset \Omega$ and let $u \in W^{1,2}_0(B(x, \varepsilon))$. Then we have $L(u,\Omega)=\int\limits_{\Omega} \varphi(u(x), Du(x)) \, dx =\int\limits_{B(x, \varepsilon)} \varphi(u(x), Du(x)) \, dx$. With
\begin{align*}
 \Gamma(L^2(\Omega, \R^N))-\lim\limits_{q \rightarrow \infty} L^{n_{k_q}}(u, \Omega) \leq \liminf\limits_{q \rightarrow \infty} L^{n_{k_q}}(u^{n_{k_q}}, \Omega) = \liminf\limits_{q \rightarrow \infty} L^{n_{k_q}}(u^{n_{k_q}}, \Omega \setminus \partial B(x, \varepsilon))
\end{align*}
for an arbitrary $u^n \rightarrow u$ in $L^2(\Omega, \R^N)$ and
\begin{align*}
 \Gamma(L^2(\Omega, \R^N))-\lim\limits_{q \rightarrow \infty} L^{n_{k_q}}(u, \Omega) \geq \limsup\limits_{q \rightarrow \infty} L^{n_{k_q}}(v^{n_{k_q}}, \Omega) = \limsup\limits_{q \rightarrow \infty} L^{n_{k_q}}(v^{n_{k_q}}, \Omega \setminus \partial B(x, \varepsilon))
\end{align*}
for the recovery sequence $v^n$ and the uniqueness of the $\Gamma$-limit, we get
\begin{align*}
 \Gamma(L^2(\Omega, \R^N))-\lim\limits_{q \rightarrow \infty} L^{n_{k_q}}(u, \Omega \setminus \partial B(x, \varepsilon))=\Gamma(L^2(\Omega, \R^N))-\lim\limits_{q \rightarrow \infty} L^{n_{k_q}}(u, \Omega).
\end{align*}
On the other hand, we then get
\begin{align*}
 L(u,\Omega)&=\Gamma(L^2(\Omega, \R^N))-\lim\limits_{q \rightarrow \infty} L^{n_{k_q}}(u, \Omega \setminus \partial B(x, \varepsilon)) \\
&= F(0,\Omega \setminus \overline{B(x, \varepsilon)}) + \int\limits_{B(x, \varepsilon)} g(y,u(y),Du(y)) \, dy = \int\limits_{B(x, \varepsilon)} g(y,u(y),Du(y)) \, dy
\end{align*}
since with Remark \ref{bembounded}
\begin{align*}
 0 \leq F(0,\Omega \setminus \overline{B(x, \varepsilon)}) & \leq \liminf\limits_{q \rightarrow \infty} L^{n_{k_q}}(0, \Omega \setminus \overline{B(x, \varepsilon)})  \leq M^2 m^2 N^2 \lVert 0 \rVert^2_{L^2(\Omega \setminus \overline{B(x, \varepsilon)}, \R^N)}=0.
\end{align*}
Altogether, this yields $\int\limits_{B(x, \varepsilon)} \varphi(u(y), Du(y)) \, dy= \int\limits_{B(x, \varepsilon)} g(y,u(y),Du(y)) \, dy \nonumber$ and, by letting $\varepsilon \rightarrow 0$, we get $\varphi(u(x), Du(x))=g(x,u(x),Du(x))$ almost everywhere in $\Omega$. Thus, for every $u \in W^{1,2}(\Omega, \R^N)$ and $U \in \mathcal{A}(\Omega)$, every subsequence of $L^n(u,U)$ has another subsequence $\Gamma(L^2(\Omega, \R^N))$-converging to $L(u,U)$ and by Urysohn's property of $\Gamma$-convergence (\cite[Theorem 1.44]{Braides}), the whole sequence also $\Gamma(L^2(\Omega, \R^N))$-converges to $L(u, U)$.
\end{proof}
The main reason why we cannot approximate all of the desired metrics with such sequences is that metrics independent of $x$ and $u(x)$ can be approximated by Riemannian metrics independent of $x$ and $u(x)$ if they can be approximated at all, as we can see in the following lemma and Lemma \ref{indepb}.

\begin{lemma}\label{indepa}
 Let $L^n : W^{1,2}(\Omega, \R^N) \times \mathcal{A}(\Omega) \rightarrow [0,\infty)$ be a sequence of functionals defined by $L^n(u,B):=\int\limits_{B} a_n^{\alpha \beta}(x) b_{ij}^n (u(x)) u_{\alpha}^i (x) u_{\beta}^j(x) \text{ } dx \nonumber$ satisfying the conditions \eqref{bounded}, \eqref{growth}, \eqref{UC} and \eqref{fusco} so that $L^n(u, \Omega)$ $\Gamma (L^2(\Omega,\R^N))$-converges to $L(u, \Omega)$ for every $u \in W^{1,2}(\Omega, \R^N)$, where the functional $L: W^{1,2}(\Omega, \R^N) \times \mathcal{A}(\Omega) \rightarrow [0,\infty)$ is defined by $L(u, B):=\int\limits_{B} \varphi(u(x), Du(x)) \text{ } dx$ for a non-negative function $\varphi$ satisfying the growth condition $\varphi(s,A) \leq c_2 \lvert A \rvert^2$ for some $c_2>0$ and for all $s \in \R^N, A \in \R^{N \times m}$. Then the sequence $K^n$ defined by $K^n(u):=\int\limits_{\Omega} a_n^{\alpha \beta}(x_0) b_{ij}^n (u(x)) u_{\alpha}^i (x) u_{\beta}^j(x) \, dx$ $\Gamma(L^2(\Omega, \R^N))$-converges to $L(u, \Omega)$ as well for all $u \in W^{1,2}(\Omega, \R^N)$. Note that in particular, the coefficients $a_n^{\alpha \beta}(x_0)$ are chosen independent of $x$.
 \end{lemma}

\begin{proof}
 Let $k>0$ and define $M_k=x_0+\frac{1}{k} \Z^m$. Then define $\Tilde{L}_k^n$ by
\begin{align*}
 \Tilde{L}_k^n(u,B):=\sum\limits_{z \in M_k} \int\limits_{\left(z+(-\frac{1}{2k}, \frac{1}{2k})^m\right) \cap B} \Tilde{a}^{\alpha \beta}_{n,k}(x) b_{ij}^n(u(x)) u^i_{\alpha}(x) u^j_{\beta}(x) \, dx,
\end{align*}
where $\Tilde{a}^{\alpha \beta}_{n,k} (x)=a_n^{\alpha \beta}(x-z+x_0)$ for $x \in z+(-\frac{1}{2k}, \frac{1}{2k})^m$, $z \in M_k$. By Remark \ref{bembounded} and \cite[Proposition 12.3]{braidef}, for every subsequence $\Tilde{L}_k^{n_l}$ we obtain the existence of a subsequence $\Tilde{L}_k^{n_{l_q}}$ so that $\Gamma(L^2(\Omega, \R^N))-\lim\limits_{q \rightarrow \infty} \Tilde{L}_k^{n_{l_q}}(u,U)$ exists for every $u \in W^{1,2}(\Omega, \R^N)$ and every $U \in \mathcal{A}(\Omega)$ and that
\begin{align*}
 &\Gamma(L^2(\Omega, \R^N))-\lim\limits_{q \rightarrow \infty} \Tilde{L}^{n_{l_q}}_k(u, \Omega)= \sum\limits_{z \in M_k} \Gamma(L^2( \Omega, \R^N ))-\lim\limits_{q \rightarrow \infty} \Tilde{L}^{n_{l_q}}_k\left(u, \left(z+(-\frac{1}{2k}, \frac{1}{2k})^m\right) \cap \Omega\right).
\end{align*}
By defining $v_z(x):=u(x+z-x_0)$ and using Lemma \ref{local}, $\Gamma(L^2(\Omega, \R^N))-\lim\limits_{q \rightarrow \infty} \Tilde{L}^{n_{l_q}}_k(u, \Omega)$ equals
\allowdisplaybreaks\begin{align*}
 & \sum\limits_{z \in M_k} \Gamma(L^2 (\Omega-z+x_0, \R^N ) )-\lim\limits_{q \rightarrow \infty} \int\limits_{\mathclap{\left(x_0+(-\frac{1}{2k}, \frac{1}{2k})^m\right) \cap (\Omega-z+x_0)}} a^{\alpha \beta}_{n_{l_q}}(x) b_{ij}^{n_{l_q}}(v_z(x)) (v_z)^i_{\alpha}(x) (v_z)^j_{\beta}(x) \, dx\\
 =& \sum\limits_{z \in M_k} \int\limits_{\left(x_0+(-\frac{1}{2k}, \frac{1}{2k})^m\right) \cap (\Omega-z+x_0)} \varphi(v_z(x), Dv_z(x)) \, dx\\
 =& \sum\limits_{z \in M_k} \int\limits_{\left(z+(-\frac{1}{2k}, \frac{1}{2k})^m\right) \cap \Omega} \varphi(u(x), Du(x)) \, dx =\, L(u, \Omega),
\end{align*}
so for every subsequence $\Tilde L_k^{n_l}$ of $\Tilde L_k^n$ there exists a further subsequence $\Tilde L_k^{n_{l_q}}$ so that $\Tilde L_k^{n_{l_q}}(\cdot, \Omega)$ $\Gamma(L^2(\Omega, \R^N))$-converges to $L(\cdot, \Omega)$. By Urysohn's property of $\Gamma$-convergence (\cite[Theorem 1.44]{Braides}), the whole sequence $\Tilde L_k^n(\cdot, \Omega)$ also $\Gamma(L^2(\Omega, \R^N))$-converges to $L(\cdot, \Omega)$ for all $k>0$. Then choose an arbitrary $u \in W^{1,2}(\Omega, \R^N)$ and a sequence $u^n$ in $W^{1,2}(\Omega, \R^N)$ converging in $L^2(\Omega, \R^N)$ to $u$. Let $u^{n_l}$ be the subsequence of $u^n$ satisfying $\liminf\limits_{n \rightarrow \infty} K^n(u^n)=\lim\limits_{l \rightarrow \infty} K^{n_l}(u^{n_l})$. Then if $u^{n_l}$ is not bounded in $W^{1,2}(\Omega, \R^N)$, we have $\liminf\limits_{n \rightarrow \infty} K^n(u^n)=\lim\limits_{l \rightarrow \infty} K^{n_l}(u^{n_l}) \geq \limsup\limits_{l \rightarrow \infty} c_1 \lVert Du^{n_l} \rVert_{L^2(\Omega)}^2 = \infty$, since $\lVert u^{n_l} \rVert_{L^2(\Omega)}$ is clearly bounded. So the $\liminf$-inequality is satisfied for $K^n$. Otherwise, choose an arbitrary $\Tilde{\varepsilon}>0$ and define $\varepsilon:=\tfrac{\Tilde{\varepsilon}}{M N^2 m^2 (\sup\limits_{l \in \N} \lVert u^{n_l} \rVert^2_{W^{1,2}(\Omega)}+1)}$. Choose $k$ large enough so that
\begin{align*}
 \lvert a_{n_l}^{\alpha \beta}(x)-a_{n_l}^{\alpha \beta}(x_0) \rvert < \varepsilon \quad \text{for almost every } x \in x_0 + \left(-\frac{1}{2k},\frac{1}{2k} \right)^m \quad \forall \, l \in \N, \, \alpha, \beta \in \left\{1,...,m\right\}.
\end{align*}
Then, $\lvert \Tilde{L}_k^{n_l}(u^{n_l}, \Omega) - K^{n_l}(u^{n_l}) \rvert$ is less than or equal to
\begin{align*}
 &\sum\limits_{z \in M_k} \int\limits_{z+(-\frac{1}{2k}, \frac{1}{2k})^m \cap \Omega} \lvert a_{n_l}^{\alpha \beta}(x-z+x_0)-a_{n_l}^{\alpha \beta}(x_0) \rvert \cdot \lvert b_{ij}^{n_l} (u^{n_l}(x)) \rvert \cdot \lvert (u^{n_l})^i_{\alpha}(x) \rvert \cdot \lvert (u^{n_l})^j_{\beta}(x) \rvert \, dx \\
 \leq& \, M \varepsilon N^2 m^2 \lVert Du^{n_l} \rVert_{L^2(\Omega)}^2 \leq \, \Tilde{\varepsilon}.
 \end{align*}
Now suppose there is $\bar{N} \in \N$ so that $L(u,\Omega)-2 \Tilde{\varepsilon} > K^{n_l}(u^{n_l})$ for all $l >\bar{N}$. Then with the same $k$ as above, we have $K^{n_l}(u^{n_l}) \geq \Tilde{L}_k^{n_l}(u^{n_l}, \Omega) - \Tilde{\varepsilon}$. This implies $\Tilde{L}_k^{n_l}(u^{n_l}, \Omega)<L(u, \Omega)-\Tilde{\varepsilon}$ for all $l>\bar{N}$ which is a contradiction to the $\liminf$-inequality $\liminf\limits_{n \rightarrow \infty} \Tilde{L}_k^n(u^n, \Omega) \geq L(u, \Omega)$. Thus, for every $\Tilde{\varepsilon}>0$ and $\bar{N} \in \N$, we find $l>\bar{N}$ so that $K^{n_l}(u^{n_l}) \geq L(u, \Omega)-2\Tilde{\varepsilon}$. This implies $\liminf\limits_{n \rightarrow \infty} K^n(u^n) =\lim\limits_{l \rightarrow \infty} K^{n_l}(u^{n_l})\geq L(u, \Omega)$ which is the $\liminf$-inequality. For the $\limsup$-inequality let $u^n$ be the recovery sequence for $\Tilde{L}^n_k$. If $u^n$ is not bounded in $W^{1,2}(\Omega, \R^N)$ we have $L(u, \Omega) \geq \limsup\limits_{n \rightarrow \infty} \Tilde{L}^n_k(u^n, \Omega) \geq \limsup\limits_{n \rightarrow \infty} c_1 \lVert Du^n \rVert_{L^2(\Omega)}^2 = \infty$. Thus, the $\limsup$-inequality clearly holds for $K^n$. Otherwise, let $u^{n_l}$ be the subsequence of $u^n$ satisfying $\limsup\limits_{n \rightarrow \infty} K^n(u^n)=\lim\limits_{l \rightarrow \infty} K^{n_l}(u^{n_l})$ and suppose there is $\bar{N} \in \N$ so that $L(u,\Omega)+2 \Tilde{\varepsilon} < K^{n_l}(u^{n_l})$ for all $l >\bar{N}$. With the same computation as above, we get $K^{n_l}(u^{n_l})\leq \Tilde{L}_k^{n_l}(u^{n_l}, \Omega) + \Tilde{\varepsilon}$ for $k$ large enough. This implies $\Tilde{L}_k^{n_l}(u^{n_l}, \Omega)>L(u, \Omega)+\Tilde{\varepsilon}$ for all $l>\bar{N}$ which is a contradiction to the $\limsup$-inequality $\limsup\limits_{n \rightarrow \infty} \Tilde{L}_k^n(u^n, \Omega) \leq L(u, \Omega)$. Thus, for every $\Tilde{\varepsilon}>0$ and $\bar{N} \in \N$, we find $l>\bar{N}$ so that $K^{n_l}(u^{n_l}) \leq L(u, \Omega)+2\Tilde{\varepsilon}$. This implies $\liminf\limits_{l \rightarrow \infty} K^{n_l}(u^{n_l}) \leq L(u, \Omega)$ and so we get $\limsup\limits_{n \rightarrow \infty} K^n(u^n)=\liminf\limits_{l \rightarrow \infty} K^{n_l}(u^{n_l}) \leq L(u, \Omega)$ which is the $\limsup$-inequality for the recovery sequence $u^n$. Altogether, we have now proven that $K^n$ $\Gamma(L^2(\Omega, \R^N))$-converges to $L(\cdot, \Omega)$. 
\end{proof}

\begin{proposition}\label{measure}
 Let $L^n$ be a sequence satisfying \eqref{bounded}, \eqref{growth} and let every subsequence of $L^n$ satisfy \eqref{explrecseq}. Then there exists a subsequence $L^{n_k}$ so that $F(u,E):=\Gamma(L^2(\Omega, \R^N))-\lim\limits_{k \rightarrow \infty} L^{n_k}(u,E)$ exists for all $u \in W^{1,2}(\Omega,\R^N)$ and every Borel set $E \subset \Omega$ and $F(u, \cdot)$ is a Borel measure for every $u \in W^{1,2}(\Omega, \R^N)$.
\end{proposition}

\begin{proof}
 By \cite[Proposition 12.3]{braidef}, there exists a subsequence $L^{n_k}$ of $L^n$ so that $F(u,U)=\Gamma-\lim\limits_{k \rightarrow \infty} L^{n_k}(u,U)$ exists for all $U \in \mathcal{A}(\Omega)$ and $u \in W^{1,2}(\Omega, \R^N)$ and $F(u, \cdot)$ is the restriction of a Borel measure $\nu_u$ to $\mathcal{A}(\Omega)$ for every $u \in W^{1,2}(\Omega, \R^N)$. Let $E \subset \Omega$ be a Borel set. Let $L^{n_{k_l}}$ be an arbitrary subsequence of $L^{n_k}$. Then by \cite[Proposition 7.9]{braidef} there exists a further subsequence $L^{n_{k_{l_q}}}$ so that $\Gamma(L^2(\Omega, \R^N))-\lim\limits_{q \rightarrow \infty} L^{n_{k_{l_q}}}(u,E)$ exists for all $u \in W^{1,2}(\Omega, \R^N)$. Let $u \in W^{1,2}(\Omega, \R^N)$. Choose $\varepsilon>0$ and $U \in \mathcal{A}(\Omega)$ and $u^{n_{k_{l_q}}} \rightarrow u$ in $L^2(\Omega, \R^N)$ so that $E \subset U$, $\Gamma(L^2(\Omega, \R^N))-\limsup\limits_{q \rightarrow \infty} L^{n_{k_{l_q}}}(u,E)=\limsup\limits_{q \rightarrow \infty} L^{n_{k_{l_q}}}(u^{n_{k_{l_q}}},E)$ and that $L^{n_{k_{l_q}}}(u^{n_{k_{l_q}}},E) \geq L^{n_{k_{l_q}}}(u^{n_{k_{l_q}}},U) - \varepsilon$ for all $q \in \N$. Thus, we achieve
\begin{align*}
 &\Gamma(L^2(\Omega,\R^N))-\lim\limits_{q \rightarrow \infty} L^{n_{k_{l_q}}}(u,E) \geq \limsup\limits_{q \rightarrow \infty} L^{n_{k_{l_q}}}(u^{n_{k_{l_q}}},E) \geq \limsup\limits_{q \rightarrow \infty} L^{n_{k_{l_q}}}(u^{n_{k_{l_q}}},U) - \varepsilon \\
\geq& \liminf\limits_{q \rightarrow \infty} L^{n_{k_{l_q}}}(u^{n_{k_{l_q}}},U) - \varepsilon \geq \nu_u(U) - \varepsilon \geq \nu_u(E) - \varepsilon
\end{align*}
and by the arbitrariness of $\varepsilon$ we deduce $\Gamma(L^2(\Omega,\R^N))-\lim\limits_{q \rightarrow \infty} L^{n_{k_{l_q}}}(u,E) \geq \nu_u(E)$. Now choose $\varepsilon>0$ and $U \in \mathcal{A}(\Omega)$ so that $E \subset U$ and $\nu_u(U) \leq \nu_u(E)+\varepsilon$, which is possible by the regularity of Borel measures on Polish spaces \cite[Theorem 1.16, p 320]{Elstrodt}. Then for the recovery sequence $u^{n_{k_{l_q}}}$ in $U$, we get
\begin{align*}
 &\Gamma(L^2(\Omega,\R^N))-\lim\limits_{q \rightarrow \infty} L^{n_{k_{l_q}}}(u,E) \leq \liminf\limits_{q \rightarrow \infty} L^{n_{k_{l_q}}}(u^{n_{k_{l_q}}},E) \\
\leq &\limsup\limits_{q \rightarrow \infty} L^{n_{k_{l_q}}}(u^{n_{k_{l_q}}},U) \leq \nu_u(U) \leq \nu_u(E)+\varepsilon 
\end{align*}
and by the arbitrariness of $\varepsilon$ we deduce $\Gamma(L^2(\Omega,\R^N))-\lim\limits_{q \rightarrow \infty} L^{n_{k_{l_q}}}(u,E) \leq \nu_u(E)$, which implies $\Gamma(L^2(\Omega,\R^N))-\lim\limits_{q \rightarrow \infty} L^{n_{k_{l_q}}}(u,E) = \nu_u(E)$. Thus, for every subsequence $L^{n_{k_l}}$ of $L^{n_k}$ there exists a further subsequence $L^{n_{k_{l_q}}}$ so that $\Gamma(L^2(\Omega,\R^N))-\lim\limits_{q \rightarrow \infty} L^{n_{k_{l_q}}}(u,E) = \nu_u(E)$. By Urysohn's property of $\Gamma$-convergence (\cite[Theorem 1.44]{Braides}), the whole sequence $L^{n_k}(u, E)$ also $\Gamma(L^2(\Omega, \R^N))$-converges to $\nu_u(E)$, which concludes the proof.
\end{proof}

\begin{lemma}\label{local2}
 Let $L^n(u,\Omega)$ be a sequence of energy functionals $\Gamma(L^2(\Omega, \R^N))$-converging to the energy functional $L(u,\Omega)=\int\limits_{\Omega} \varphi(u(x), Du(x)) \, dx$ for a non-negative function $\varphi$ satisfying the growth condition $\varphi(s,A) \leq c_2 \lvert A \rvert^2$ for some $c_2>0$ and for all $s \in \R^N, A \in \R^{N \times m}$. Let $L^n$ satisfy the conditions \eqref{bounded}, \eqref{growth}, \eqref{fusco}, and let every subsequence of $L^n$ satisfy \eqref{explrecseq}. Then $L^n(u,E)$ $\Gamma(L^2(\Omega, \R^N))$-converges to $L(u,E)$ for every Borel set $E \subset \Omega$ and every $u \in W^{1,2}(\Omega, \R^N)$.
\end{lemma}

\begin{proof}
 By Lemma \ref{local}, $L^n(u,U)$ $\Gamma(L^2(\Omega, \R^N))$-converges to $L(u,U)$ for every $U \in \mathcal{A}(\Omega)$. Let $E \subset \Omega$ be a Borel set. Let $L^{n_k}$ be an arbitrary subsequence of $L^n$. Then by \cite[Proposition 7.9]{braidef}, there exists a further subsequence $L^{n_{k_l}}$ so that $\Gamma(L^2(\Omega, \R^N))-\lim\limits_{l \rightarrow \infty} L^{n_{k_l}}(u,E)$ exists for all $u \in W^{1,2}(\Omega, \R^N)$. Let $u \in W^{1,2}(\Omega, \R^N)$. Choose $\varepsilon>0$, $U \in \mathcal{A}(\Omega)$ and $u^{n_{k_l}} \rightarrow u$ in $L^2(\Omega, \R^N)$ so that $E \subset U$, $\Gamma(L^2(\Omega, \R^N))-\limsup\limits_{l \rightarrow \infty} L^{n_{k_l}}(u,E)=\limsup\limits_{l \rightarrow \infty} L^{n_{k_l}}(u^{n_{k_l}},E)$ and that $L^{n_{k_l}}(u^{n_{k_l}},E) \geq L^{n_{k_l}}(u^{n_{k_l}},U) - \varepsilon$ for all $l \in \N$. Thus, in the same way as in the proof of Lemma \ref{measure} we achieve
\begin{align*}
 \Gamma(L^2(\Omega,\R^N))-\lim\limits_{l \rightarrow \infty} L^{n_{k_l}}(u,E) \geq L(u,E) - \varepsilon
\end{align*}
and by the arbitrariness of $\varepsilon$ we deduce $\Gamma(L^2(\Omega,\R^N))-\lim\limits_{l \rightarrow \infty} L^{n_{k_l}}(\cdot,E) \geq L(\cdot,E)$. Now choose $\varepsilon>0$ and $U \in \mathcal{A}(\Omega)$ so that $E \subset U$ and $L(u,U) \leq L(u,E)+\varepsilon$. Such a set $U$ exists because by the regularity of the Lebesgue measure we have $\mathcal{L}(E)=\inf \{ \mathcal{L}(U); U \in \mathcal{A}(\Omega), E \subset U \}$. Now let $U_l$ be a sequence in $\mathcal{A}(\Omega)$ so that $E \subset U_l$ for every $l \in \N$ and $\lim\limits_{l \rightarrow \infty} \mathcal{L}(U_l)=\mathcal{L}(E)$. Then we have
\begin{align*}
 \lim\limits_{l \rightarrow \infty} L(u,U_l \setminus E) =\lim\limits_{l \rightarrow \infty} \int\limits_{\Omega} \varphi(u(x),Du(x)) \cdot \mathbbm{1}_{U_l \setminus E}(x) \, dx =0
\end{align*}
by Lebesgue's Dominated Convergence Theorem (\cite[Theorem 1.8]{lieblos}). Now choose $L \in \N$ so that $L(u, U_L \setminus E) < \varepsilon$. Then with $U=U_L$ we achieve $L(u,U) = L(u,E) + L(u, U \setminus E) < L(u,E) + \varepsilon$. Then for the recovery sequence $u^{n_{k_l}}$ in $U$, we achieve $\Gamma(L^2(\Omega,\R^N))-\lim\limits_{l \rightarrow \infty} L^{n_{k_l}}(u,E) \leq L(u,E)+\varepsilon$ as in the proof of Lemma \ref{measure} and by the arbitrariness of $\varepsilon$ we deduce $\Gamma(L^2(\Omega,\R^N))-\lim\limits_{l \rightarrow \infty} L^{n_{k_l}}(\cdot,E) \leq L(\cdot,E)$, which implies $\Gamma(L^2(\Omega,\R^N))-\lim\limits_{l \rightarrow \infty} L^{n_{k_l}}(\cdot,E) = L(\cdot,E)$. Thus, for every subsequence $L^{n_k}$ of $L^n$ there exists a further subsequence $L^{n_{k_l}}$ so that $\Gamma(L^2(\Omega,\R^N))-\lim\limits_{l \rightarrow \infty} L^{n_{k_l}}(\cdot,E) = L(\cdot,E)$. By Urysohn's property of $\Gamma$-convergence (\cite[Theorem 1.44]{Braides}), the whole sequence $L^n(\cdot, E)$ also $\Gamma(L^2(\Omega, \R^N))$-converges to $L(\cdot, E)$, which concludes the proof.
\end{proof}

\begin{korollar}\label{recseq}
  Let $L^n(u,\Omega)$ be a sequence of energy functionals $\Gamma(L^2(\Omega, \R^N))$-converging to the energy functional $L(u,\Omega)=\int\limits_{\Omega} \varphi(u(x), Du(x)) \, dx$ for a non-negative function $\varphi$ satisfying the growth condition $\varphi(s,A) \leq c_2 \lvert A \rvert^2$ for some $c_2>0$ and for all $s \in \R^N, A \in \R^{N \times m}$. Let $L^n$ satisfy the conditions \eqref{bounded}, \eqref{growth}, \eqref{fusco}, let every subsequence of $L^n$ satisfy \eqref{explrecseq}, and let $u^n$ be a recovery sequence for $u$ in $\Omega$. Then $u^n$ is a recovery sequence for $u$ in $E$ for every Borel set $E \subset \Omega$.
\end{korollar}

\begin{proof}
By Proposition \ref{local2}, we get $\Gamma(L^2(\Omega, \R^N))-\lim\limits_{n \rightarrow \infty} L^n(u,E) = L(u,E)$ for every Borel set $E \subset \Omega$. Let $E \subset \Omega$ be a Borel set and assume $\limsup\limits_{n \rightarrow \infty} L^n(u^n, E) > \Gamma(L^2(\Omega, \R^N))-\lim\limits_{n \rightarrow \infty} L^n(u,E)=L(u,E)$. This implies
\begin{align*}
&\Gamma(L^2(\Omega, \R^N))-\lim\limits_{n \rightarrow \infty} L^n(u, \Omega \setminus E) = L(u, \Omega )-L(u, E) \\
 \geq &\limsup\limits_{n \rightarrow \infty} L^n(u^n, E) + \liminf\limits_{n \rightarrow \infty} L^n(u^n, \Omega \setminus E) -L(u, E)> \liminf\limits_{n \rightarrow \infty} L^n(u^n, \Omega \setminus E) 
\end{align*}
which is a contradiction to the $\liminf$-inequality in $\Omega \setminus E$.

\end{proof}

\begin{bemerkung}\label{recseqgamma}
 In Corollary \ref{recseq}, the main requirement is the $\Gamma(L^2(\Omega, \R^N))$-convergence to $L(\cdot, E)$ for every Borel set $E$, as can be seen in the proof. So instead of requiring the conditions \eqref{bounded}, \eqref{growth}, \eqref{fusco} and \eqref{explrecseq} for every subsequence, it is enough to prescribe the $\Gamma(L^2(\Omega, \R^N))$-convergence for every Borel set $E$.
\end{bemerkung}

For the independence of values in the image of $u$ of the approximating functionals, we will need the following notations: for $u \in W^{1,2}(\Omega, \R^N)$ we define the sets
\begin{align*}
 N_k^{z,u}:=\{ x; u(x) \in \tfrac{1}{k}z+[-\tfrac{1}{2k},\tfrac{1}{2k})^N \} \quad \text{for } k>0, z \in \Z^N
\end{align*}
and, for a given sequence $L^n(u,B)=\int\limits_{B} a_n^{\alpha \beta}(x) b_{ij}^n (u(x)) u_{\alpha}^i (x) u_{\beta}^j(x) \, dx$, we define
\begin{align*}
 L^n_{k,u}(v,B):=\sum\limits_{z \in \Z^N} \int\limits_{N_k^{z,u} \cap B} a_n^{\alpha \beta}(x) b_{ij}^n (v(x)-\tfrac{1}{k}z) v^i_{\alpha}(x) v^j_{\beta}(x) \, dx.
\end{align*}

\begin{bemerkung}\label{Vererbung}
 If the sequence $L^n$ is defined by $L^n(v):=\int\limits_{\Omega} a_n^{\alpha \beta} b_{ij}^n v^i_{\alpha}(x) v_j^{\beta}(x) \, dx$, we have $L^n(v,B)=L^n_{k,u}(v,B)$. Hence, by Corollary \ref{exfus}, $L^n$ and $L^n_{k,u}$ satisfy the condition \eqref{explrecseq}. Furthermore, the conditions \eqref{bounded}, \eqref{growth} and \eqref{UC} clearly are inherited by $L^n_{k,u}$ if $L^n$ satisfies these conditions. 
\end{bemerkung}

\begin{lemma}\label{indepb}
 Let $L^n : W^{1,2}(\Omega, \R^N) \times \mathcal{A}(\Omega) \rightarrow [0,\infty)$ be a sequence of functionals defined by $L^n(u,B):=\int\limits_{B} a_n^{\alpha \beta}(x) b_{ij}^n (u(x)) u_{\alpha}^i (x) u_{\beta}^j(x) \text{ } dx$ satisfying the conditions \eqref{bounded}, \eqref{growth}, \eqref{UC} and \eqref{fusco}, so that every subsequence of $L^n$ and every subsequence of $L^n_{k,u}$ satisfies \eqref{explrecseq} for every $k>0$, $u \in W^{1,2}(\Omega, \R^N)$ and so that $L^n(u, \Omega)$ $\Gamma (L^2(\Omega,\R^N))$-converges to $L(u, \Omega)$ for every $u \in W^{1,2}(\Omega, \R^N)$, where the functional $L: W^{1,2}(\Omega, \R^N) \times \mathcal{A}(\Omega) \rightarrow [0,\infty)$ is defined by $L(u, B):=\int\limits_{B} \varphi(Du(x)) \text{ } dx$ for a non-negative function $\varphi$ satisfying the growth condition $\varphi(s,A) \leq c_2 \lvert A \rvert^2$ for some $c_2>0$ and for all $s \in \R^N, A \in \R^{N \times m}$. Then for the sequence $K^n$ defined by $K^n(u,B):=\int\limits_{B} a_n^{\alpha \beta}(x) b_{ij}^n (0) u_{\alpha}^i (x) u_{\beta}^j(x) \, dx$ $\Gamma(L^2(\Omega, \R^N))-\lim\limits_{n \rightarrow \infty} K^n(u, \Omega)$ exists for every $u \in W^{1,2}(\Omega, \R^N)$ and $\Gamma(L^2(\Omega, \R^N))-\lim\limits_{n \rightarrow \infty} K^n(u, \Omega)=L(u, \Omega)$ for every $u \in W^{1,2}(\Omega, \R^N)$. Note that in particular, the coefficients $b_{ij}^n(0)$ are chosen independent of $u(x)$.
 \end{lemma}

\begin{proof}
Choose an arbitrary function $u \in W^{1,2}(\Omega, \R^N)$. As in \cite{Borel}, we can choose a representative of $u$ so that $u_i$ is a Borel function for every $i \in \{1,..,m\}$. Thus, $N_k^{z,u}$ is a Borel set for all $k>0$, $z \in \Z^N$, since it is the intersection of Borel sets. By Remark \ref{Vererbung} and Proposition \ref{measure}, for every subsequence $L_{k,u}^{n_l}$ we obtain the existence of a subsequence $L_{k,u}^{n_{l_q}}$ so that $\Gamma(L^2(\Omega, \R^N))-\lim\limits_{q \rightarrow \infty} L_{k,u}^{n_{l_q}}(v,E)$ exists for every Borel set $E \subset \Omega$ and that
\begin{align*}
 \Gamma(L^2(\Omega, \R^N))-\lim\limits_{q \rightarrow \infty} L^{n_{l_q}}_{k,u}(v, E) = \sum\limits_{z \in \Z^N} \Gamma(L^2( \Omega, \R^N ))-\lim\limits_{q \rightarrow \infty} L^{n_{l_q}}_{k,u}(v, N_k^{z,u}  \cap E)
\end{align*}
for every Borel set $E \subset \Omega$. With $v$ fixed and $w_z(x):=v(x)-\frac{1}{k}z$, we observe with Lemma \ref{local2} that $\Gamma(L^2(\Omega, \R^N))-\lim\limits_{q \rightarrow \infty} L^{n_{l_q}}_{k,u}(v,E)$ equals
\begin{align*}
 \sum\limits_{z \in \Z^N} \Gamma(L^2(\Omega, \R^N))-\lim\limits_{q \rightarrow \infty} L^{n_{l_q}}(w_z, N_k^{z,u} \cap E) =\sum\limits_{z \in \Z^N} \int\limits_{N_k^{z,u} \cap E} \varphi(Dw_z(x)) \, dx = \, L(v,E),
\end{align*}
so for every subsequence $L_{k,u}^{n_l}$ of $L_{k,u}^n$ there exists a further subsequence $L_{k,u}^{n_{l_q}}$ so that $L_{k,u}^{n_{l_q}}(\cdot, E)$ $\Gamma(L^2(\Omega, \R^N))$-converges to $L(\cdot, E)$ for every Borel set $E \subset \Omega$. By Urysohn's property of $\Gamma$-convergence (\cite[Theorem 1.44]{Braides}), the whole sequence $L_{k,u}^n(\cdot, E)$ also $\Gamma(L^2(\Omega, \R^N))$-converges to $L(\cdot, E)$ for all $k>0$. Let $u^n$ be an arbitrary sequence in $W^{1,2}(\Omega, \R^N)$ converging to $u$ in $L^2(\Omega, \R^N)$. Then let $u^{n_l}$ be a subsequence satisfying $\lim\limits_{l \rightarrow \infty} K^{n_l}(u^{n_l}, \Omega)=\liminf\limits_{n \rightarrow \infty} K^n(u^n, \Omega)$ and $u^{n_l}(x) \rightarrow u(x)$ for almost every $x \in \Omega$. Then by Egorov's Theorem (\cite[Theorem 5.3, p 252]{Elstrodt}) and the regularity of the Lebesgue measure, for every $\varepsilon>0$ there is an open set $A_{\varepsilon}$ satisfying $\lvert \Omega \setminus A_{\varepsilon} \rvert > \lvert \Omega \rvert - \varepsilon$ so that $u^{n_l}$ uniformly converges to $u$ in $\Omega \setminus A_{\varepsilon}$. Then if $u^{n_l}$ is not bounded in $W^{1,2}(\Omega, \R^N)$, we have
\begin{align*}
 \liminf\limits_{n \rightarrow \infty} K^n(u^n, \Omega)=\lim\limits_{l \rightarrow \infty} \int\limits_{\Omega} a_{n_l}^{\alpha \beta}(x) b_{ij}^{n_l}(0) (u^{n_l})^i_{\alpha}(x) (u^{n_l})^j_{\beta}(x) \, dx \geq \limsup\limits_{l \rightarrow \infty} c_1 \lVert Du^{n_l} \rVert_{L^2(\Omega)}^2 = \infty,
\end{align*}
since $\lVert u^{n_l} \rVert_{L^2(\Omega)}$ is clearly bounded. Thus, we can assume that $u^{n_l}$ is bounded in $W^{1,2}(\Omega, \R^N)$. Then choose an arbitrary $\Tilde{\varepsilon}>0$ and $\varepsilon<\tfrac{\Tilde{\varepsilon}}{\sup\limits_{l \in \N} \lVert Du_{n_l} \rVert_{L^2(\Omega)}^2+\mathcal{L}(\Omega)}$ so that $\int\limits_{A_{\varepsilon}} c_2 \lvert Du(x) \rvert^2 \, dx < \frac{\Tilde{\varepsilon}}{2}$. This is possible since $\lvert A_{\varepsilon} \rvert \rightarrow 0$ for $\varepsilon \rightarrow 0$ and thus, by Lebesgue's Dominated Convergence Theorem (\cite[Theorem 1.8]{lieblos}), we have $\lim\limits_{\varepsilon \rightarrow 0} \int\limits_{A_{\varepsilon}} c_2 \lvert Du (x) \rvert^2 \, dx = \lim\limits_{\varepsilon \rightarrow 0} \int\limits_\Omega c_2 \lvert Du(x) \rvert^2 \mathbbm{1}_{A_{\varepsilon}}(x) \, dx=0$. Then choose $k$ large enough so that $\omega(\lvert y \rvert) < \varepsilon$ for every $y \in [-\tfrac{1}{k}, \tfrac{1}{k})^N$ and $\Tilde{N}$ large enough so that for all $l>\Tilde{N}$ we have $\lvert u^{n_l}(x)-u(x) \vert < \frac{1}{2k}$ for all $x \in \Omega \setminus A_{\varepsilon}$. We observe for $l>\Tilde{N}$ that $\lvert L_{k,u}^{n_l}(u^{n_l}, \Omega \setminus A_{\varepsilon}) - K^{n_l}(u^{n_l}, \Omega \setminus A_{\varepsilon}) \rvert$ is less than or equal to
\begin{align*}
 \sum\limits_{z \in \Z^N} \int\limits_{N_k^{z,u} \cap (\Omega \setminus A_{\varepsilon})} \omega(\lvert u^{n_l}(x)-\frac{1}{k}z \rvert) \cdot (1+\lvert Du^{n_l}(x) \rvert^2) \, dx \leq \, \varepsilon \cdot (\mathcal{L}(\Omega)+\lVert Du^{n_l}(x) \rVert_{L^2(\Omega)}^2) \leq  \, \Tilde{\varepsilon}.
\end{align*}
Now suppose there is $\bar{N} \in \N$ so that $L(u,\Omega) - 2 \Tilde{\varepsilon} > K^{n_l}(u^{n_l}, \Omega)$ for all $l>\bar{N}$. We know that for $l>\Tilde{N}$ and the same $k$ as above $K^{n_l}(u^{n_l}, \Omega \setminus A_{\varepsilon}) \geq L^{n_l}_{k,u} (u^{n_l}, \Omega \setminus A_{\varepsilon}) - \Tilde{\varepsilon}$ holds. This implies for all $l>\max\{\bar{N}, \Tilde{N}\}$
\begin{align*}
 L^{n_l}_{k,u}(u^{n_l}, \Omega \setminus A_{\varepsilon}) < L(u, \Omega) - \Tilde{\varepsilon} \leq \int\limits_{\Omega \setminus A_{\varepsilon}} \varphi(Du(x)) \, dx + c_2 \int\limits_{A_{\varepsilon}} \lvert Du(x) \rvert^2 \, dx - \Tilde{\varepsilon}  \leq L(u, \Omega \setminus A_{\varepsilon})- \frac{\Tilde{\varepsilon}}{2},
\end{align*}
which is a contradiction to the $\liminf$-inequality $L(u, \Omega \setminus A_{\varepsilon}) \leq \liminf\limits_{n \rightarrow \infty} L^n_{k,u} (u^n, \Omega \setminus A_{\varepsilon})$. Thus, for every $\Tilde{\varepsilon}>0$ and $\bar{N} \in \N$, we find $l>\bar{N}$ so that $K^{n_l}(u^{n_l}, \Omega) \geq L(u, \Omega)-2\Tilde{\varepsilon}$. This implies $\liminf\limits_{n \rightarrow \infty} K^n(u^n, \Omega) =\lim\limits_{l \rightarrow \infty} K^{n_l}(u^{n_l}, \Omega) \geq L(u, \Omega)$, which is the $\liminf$-inequality. For the $\limsup$-inequality let $u^n$ be a recovery sequence for $L^n_{k,u}(u, \Omega)$. If $u^n$ is not bounded in $W^{1,2}(\Omega, \R^N)$, we have $L(u, \Omega) \geq \limsup\limits_{n \rightarrow \infty} L^n_{k,u}(u^n, \Omega) \geq \limsup\limits_{n \rightarrow \infty} c_1 \lVert Du^n \rVert^2_{L^2(\Omega)} = \infty$. Thus, the $\limsup$-inequality clearly holds for $K^n$. Otherwise, with the same computation as above, we get the existence of $\Tilde{N} \in \N$ so that for $l>\Tilde{N}$ and $k$ large enough $\lvert L_{k,u}^{n_l}(u^{n_l}, \Omega \setminus A_{\varepsilon}) - K^{n_l}(u^{n_l}, \Omega \setminus A_{\varepsilon}) \rvert < \Tilde{\varepsilon}$ for the subsequence $u^{n_l}$, which satisfies $\lim\limits_{l \rightarrow \infty} K^{n_l}(u^{n_l}, \Omega)=\limsup\limits_{n \rightarrow \infty} K^n(u^n, \Omega)$ and $u^{n_l}(x) \rightarrow u(x)$ for almost every $x \in \Omega$ and which is still a recovery sequence for $L^{n_l}_{k,u}(\cdot, \Omega)$. Furthermore,
\begin{align*}
 \frac{\Tilde{\varepsilon}}{2}> c_2 \int\limits_{A_{\varepsilon}} \lvert Du(x) \rvert^2 \, dx \geq L(u, A_{\varepsilon}) \geq \limsup\limits_{l \rightarrow \infty} L^{n_l}(u^{n_l}, A_{\varepsilon}) \geq c_1 \limsup\limits_{l \rightarrow \infty} \int\limits_{A_{\varepsilon}} \lvert Du^{n_l}(x) \rvert^2 \, dx 
\end{align*}
since, by Corollary \ref{recseq} and Remark \ref{recseqgamma}, $u^{n_l}$ is a recovery sequence for $u$ in every Borel set $E \subset \Omega$ so that
\begin{align}
 \limsup\limits_{l \rightarrow \infty} \int\limits_{A_{\varepsilon}} \lvert Du^{n_l}(x) \rvert^2 \, dx < \tfrac{\Tilde{\varepsilon}}{2 c_1}. \label{veroeff1}
\end{align}
By Corollary \ref{recseq} and Remark \ref{recseqgamma}, $u^{n_l}$ is a recovery sequence for $L^{n_l}_{k,u}(u, \Omega \setminus A_{\varepsilon})$ as well. 
Summarizing, by this and the non-negativity of $\varphi$ and by using Remark \ref{bembounded} and \eqref{veroeff1} in $A_{\varepsilon}$ we achieve
\begin{align*}
 &\limsup\limits_{n \rightarrow \infty} K^n(u^n, \Omega)=\lim\limits_{l \rightarrow \infty} K^{n_l}(u^{n_l},\Omega) \leq \limsup\limits_{l \rightarrow \infty} K^{n_l}(u^{n_l}, \Omega \setminus A_{\varepsilon}) + \limsup\limits_{l \rightarrow \infty} K^{n_l}(u^{n_l}, A_{\varepsilon})  \\
 \leq& \limsup\limits_{l \rightarrow \infty} L^{n_l}_{k,u}(u^{n_l}, \Omega \setminus A_{\varepsilon}) + \Tilde{\varepsilon} + M^2m^2 N^2 \frac{\Tilde{\varepsilon}}{2 c_1} \leq \,  L(u, \Omega) + \Tilde{\varepsilon} (1+ \tfrac{M^2m^2 N^2}{2 c_1}).
\end{align*}
By $\Tilde{\varepsilon} \rightarrow 0$, this implies the $\limsup$-inequality $\limsup\limits_{n \rightarrow \infty} K^n(u^n, \Omega) \leq L(u, \Omega)$ and thus, $K^n( u, \Omega)$ $\Gamma(L^2(\Omega, \R^N))$-converges to $L( u, \Omega)$.
\end{proof}

\begin{bemerkung}\label{Vererbung2}
 Clearly, in Lemma \ref{indepa} and Lemma \ref{indepb}, the sequence $K^n$ inherits the conditions \eqref{bounded}, \eqref{growth} and \eqref{UC} from the sequence $L^n$.
\end{bemerkung}

\section{Counterexamples for the $\Gamma$-Density of General Dirichlet Energies}\label{Counterexamples}

In Theorem \ref{bild} and Theorem \ref{urbild}, we have seen that not every functional $L \in \mathcal{E}(\Omega)$ can be approximated by energy functionals of functions $u: R_1 \rightarrow R_2$ if one of the Riemannian manifolds $R_1$ and $R_2$ is a Euclidean domain. Now we will see that under the assumptions of Chapter \ref{PropApprox} an approximation is not possible even if both Riemannian manifolds are not Euclidean domains.

\begin{satz}\label{Gegenbeispielallg}
Not every functional $L \in \mathcal{E}(\Omega)$ can be approximated by energy functionals of the form $L^n(u)=\int\limits_{\Omega} a_n^{\alpha \beta}(x) b^n_{i j} (u(x)) u^i_{\alpha}(x) u^j_{\beta}(x) \text{ } dx$ which satisfy the conditions of Lemma \ref{indepb}.
\end{satz}

\begin{korollar}\label{corGegenbeispielallg}
 Not every functional $L \in \mathcal{E}(\Omega)$ can be approximated by energy functionals of the form $L^n(u)=\int\limits_{\Omega} a_n^{\alpha \beta}(x) b^n_{i j} u^i_{\alpha}(x) u^j_{\beta}(x) \text{ } dx$ which satisfy the conditions of Lemma \ref{indepa}.
\end{korollar}

\begin{bemerkung}
 Corollary \ref{corGegenbeispielallg} restricts the approximating metrics $L^n$ to those metrics whose integrands are independent of $u(x)$. By this effort, we do not need the condition \eqref{explrecseq} any more. Since this condition is difficult to prove, it is nice to be able to omit it.
\end{bemerkung}

\begin{proof}[Proof of Theorem \ref{Gegenbeispielallg}]
Let $L$ be defined by $L(u):=\int\limits_{\Omega} \varphi(Du(x)) \text{ } dx$ and $\varphi(A)=\sum\limits_{i=1}^m \sum\limits_{j=1}^N (a^i_j)^2 - \frac{1}{2}(a^2_1)^2 - \frac{1}{2}(a^2_2)^2 + \frac{1}{2}(a^2_1+a^2_2)^2$. Note that $\varphi$ is independent of $u(x)$. Then if there were a sequence $L^n$ $\Gamma(L^2(\Omega,\R^N))$-converging to $L$ and satisfying the conditions of Lemma \ref{indepb}, with Lemma \ref{indepa} and Lemma \ref{indepb} we could find another sequence $L^n$ defined by $L^n(u):=\int\limits_{\Omega} a_n^{\alpha \beta} b^n_{i j} u^i_{\alpha}(x) u^j_{\beta}(x) \, dx$, so that $L^n(u)$ $\Gamma(L^2(\Omega,\R^N))$-converges to $L(u)$ for every $u \in W^{1,2}(\Omega, \R^N)$. By Lemma \ref{constrecseq}, we can choose the constant sequence $u$ as recovery sequence.
Now choose the functions $u(x):=x^1 \cdot e_1$, $v(x):=x^1 \cdot e_2$, $w(x):=x^2 \cdot e_1$, $z(x):=x^2 \cdot e_2$, then we get
\begin{align*}
 \int\limits_{\Omega} \varphi(Du(x)) \, dx &= \lim\limits_{n \rightarrow \infty} L^n(u)=\lim\limits_{n \rightarrow \infty} \int\limits_{\Omega} a_n^{11} b_{11}^n \, dx, \\
\int\limits_{\Omega} \varphi(Dv(x)) \, dx &= \lim\limits_{n \rightarrow \infty} L^n(v)=\lim\limits_{n \rightarrow \infty} \int\limits_{\Omega} a_n^{11} b_{22}^n \, dx, \\
\int\limits_{\Omega} \varphi(Dw(x)) \, dx &= \lim\limits_{n \rightarrow \infty} L^n(w)=\lim\limits_{n \rightarrow \infty} \int\limits_{\Omega} a_n^{22} b_{11}^n \, dx, \\
\int\limits_{\Omega} \varphi(Dz(x)) \, dx &= \lim\limits_{n \rightarrow \infty} L^n(z)=\lim\limits_{n \rightarrow \infty} \int\limits_{\Omega} a_n^{22} b_{22}^n \, dx, 
\end{align*}
but with the above choice of $\varphi$ we obtain $\varphi(Du(s)+Dw(s)) = 2$ and $\varphi(Dv(s)+Dz(s)) = 3$. Then altogether we have that $2 \cdot \mathcal{L}(\Omega)$ equals
\begin{align*}
& \int\limits_{\Omega} \varphi(Du(x)+Dw(x)) \, dx = \lim\limits_{n \rightarrow \infty} L^n(u+w) = \lim\limits_{n \rightarrow \infty} \int\limits_{\Omega} a_n^{11} b_{11}^n + a_n^{22} b_{11}^n + a_n^{12} b_{11}^n + a_n^{21}b_{11}^n \, dx \\
=& \lim\limits_{n \rightarrow \infty} \int\limits_{\Omega} \varphi(Du(x)) + \varphi(Dw(x)) + (a_n^{12} + a_n^{21})b_{11}^n \, dx = \lim\limits_{n \rightarrow \infty} \int\limits_{\Omega} 1 + 1 + (a_n^{12} + a_n^{21})b_{11}^n \, dx, 
\end{align*}
so we can see that $\lim\limits_{n \rightarrow \infty}  (a_n^{12} + a_n^{21})b_{11}^n = 0$. On the other hand, in the same way we achieve that $3 \cdot \mathcal{L}(\Omega)$ equals
\begin{align*}
  & \int\limits_{\Omega} \varphi(Dv(x)+Dz(x)) \, dx = \lim\limits_{n \rightarrow \infty} L^n(v+z) \\
=& \lim\limits_{n \rightarrow \infty} \int\limits_{\Omega} \varphi(Dv(x)) + \varphi(Dz(x)) + (a_n^{12} + a_n^{21})b_{22}^n \, dx = \lim\limits_{n \rightarrow \infty} \int\limits_{\Omega} 1 + 1 + (a_n^{12} + a_n^{21})b_{22}^n \, dx, 
\end{align*}
so we can see that $\lim\limits_{n \rightarrow \infty}  (a_n^{12} + a_n^{21})b_{22}^n = 1$. This means that, for $n$ large enough, we have $\lvert a_n^{12} + a_n^{21} \rvert \cdot \lvert b_{11}^n \rvert < \frac{1}{4}$ and $\lvert a_n^{12} + a_n^{21} \rvert \cdot \lvert b_{22}^n \rvert > \frac{3}{4}$, and thus $3 \lvert b_{11}^n \rvert < \lvert b_{22}^n \rvert$, because by Remark \ref{Vererbung}, $\lvert a_n^{12} + a_n^{21} \rvert>0$ and $\lvert a_n^{12} + a_n^{21} \rvert>0$. This implies $3 \lvert a_n^{11} \rvert \cdot \lvert  b_{11}^n \rvert < \lvert a_n^{11} \rvert \cdot \lvert b_{22}^n \rvert$ which is a contradiction to $\lim\limits_{n \rightarrow \infty} a_n^{11} b_{11}^n=1=\lim\limits_{n \rightarrow \infty} a_n^{11} b_{22}^n$, so there can be no such sequence $L^n$.
\end{proof}

\begin{proof}[Proof of Corollary \ref{corGegenbeispielallg}]
Define $L$ as in the proof of Theorem \ref{Gegenbeispielallg}. If there were a sequence $L^n$ $\Gamma(L^2(\Omega,\R^N))$-converging to $L$ and satisfying the conditions of Lemma \ref{indepa}, with Lemma \ref{indepa} we could find another sequence $L^n$ defined by $L^n(u):=\int\limits_{\Omega} a_n^{\alpha \beta} b^n_{i j} u^i_{\alpha}(x) u^j_{\beta}(x) \, dx$ so that $L^n(u)$ $\Gamma(L^2(\Omega,\R^N))$-converges to $L(u)$ for every $u \in W^{1,2}(\Omega, \R^N)$. The rest of the proof is analogous to the proof of Theorem \ref{Gegenbeispielallg}.
\end{proof}

\begin{bemerkung}
 The class in which not every functional can be approximated can even be chosen smaller than $\mathcal{E}(\Omega)$ since the counterexample holds for a functional $L$ with integrand $\varphi$ completely independent of $u(x)$.
\end{bemerkung}

In Theorem \ref{cartanwachs}, we have seen that no functional $L \in \mathcal{C}(\Omega)$ can be approximated by a sequence of elements of $\mathcal{R}_I(\Omega)$ which satisfies the condition \eqref{growth}. Now we will see that under the assumptions of Chapter \ref{PropApprox} an approximation is not possible even without demanding isotropy and without demanding that one of the Riemannian manifolds is a Euclidean domain.

\begin{satz}\label{GegenbeispielCart}
 No $L \in \mathcal{C}(\Omega)$ with a parametric integrand $\Phi(Du_1(x) \wedge Du_2(x))$ independent of $u(x)$ can be approximated by energy functionals of the form $L^n(u)=\int\limits_{\Omega} a_n^{\alpha \beta}(x) b^n_{i j} (u(x)) u^i_{\alpha}(x) u^j_{\beta}(x) \, dx$ which satisfy the conditions of Lemma \ref{indepb}.
\end{satz}

\begin{korollar}\label{corGegenbeispielCart}
 No $L \in \mathcal{C}(\Omega)$ with a parametric integrand $\Phi(Du_1(x) \wedge Du_2(x))$ independent of $u(x)$ can be approximated by energy functionals of the form $L^n(u)=\int\limits_{\Omega} a_n^{\alpha \beta}(x) b^n_{i j} u^i_{\alpha}(x) u^j_{\beta}(x) \, dx$ which satisfy the conditions of Lemma \ref{indepa}.
\end{korollar}

\begin{bemerkung}
 As in Corollary \ref{corGegenbeispielallg}, Corollary \ref{corGegenbeispielCart} restricts the approximating metrics $L^n$ to those metrics whose integrands are independent of $u(x)$. Again by this effort, we do not need the condition \eqref{explrecseq} any more. Since this condition is difficult to prove, it is nice to be able to omit it.
\end{bemerkung}

\begin{proof}[Proof of Theorem \ref{GegenbeispielCart}]
 If there were a sequence $L^n$ $\Gamma(L^2(\Omega,\R^N))$-converging to $L$ and satisfying the conditions of Lemma \ref{indepb}, with Lemma \ref{indepa} and Lemma \ref{indepb} we could find another sequence $L^n$ defined by $L^n(u):=\int\limits_{\Omega} a_n^{\alpha \beta} b^n_{i j} u^i_{\alpha}(x) u^j_{\beta}(x) \, dx$ so that $L^n(u)$ $\Gamma(L^2(\Omega,\R^N))$-converges to $L(u)$ for every $u \in W^{1,2}(\Omega, \R^N)$. By Lemma \ref{constrecseq}, we can choose the constant sequence $u$ as recovery sequence. Thus, we can deduce that $0=\int\limits_{\Omega} \Phi(0) \, dx$ can be expressed as $L (0,x_1, 0)^T= \lim\limits_{n \rightarrow \infty} \int\limits_{\Omega} a_n^{11} b_{22}^n \, dx$, but also as $L (0, 0, x_1)^T= \lim\limits_{n \rightarrow \infty} \int\limits_{\Omega} a_n^{11} b_{33}^n \, dx$, as $L (0, x_2, 0)= \lim\limits_{n \rightarrow \infty} \int\limits_{\Omega} a_n^{22} b_{22}^n \, dx$ and as $L (0, 0, x_2)^T = \lim\limits_{n \rightarrow \infty} \int\limits_{\Omega} a_n^{22} b_{33}^n \, dx$. This implies that
\begin{align*}
 L \begin{pmatrix} 0 \\ x_1 \\ x_1 \end{pmatrix} = \lim\limits_{n \rightarrow \infty} \int\limits_{\Omega} a_n^{11} b_{22}^n+a_n^{11} b_{23}^n+a_n^{11} b_{32}^n+a_n^{11} b_{33}^n \, dx = \lim\limits_{n \rightarrow \infty} \int\limits_{\Omega} a_n^{11} b_{23}^n+a_n^{11} b_{32}^n \, dx
\end{align*}
and
\begin{align*}
 L \begin{pmatrix} 0 \\ x_2 \\ x_2 \end{pmatrix} = \lim\limits_{n \rightarrow \infty} \int\limits_{\Omega} a_n^{22} b_{22}^n+a_n^{22} b_{23}^n+a_n^{22} b_{32}^n+a_n^{22} b_{33}^n \, dx = \lim\limits_{n \rightarrow \infty} \int\limits_{\Omega} a_n^{22} b_{23}^n+a_n^{22} b_{32}^n \, dx,
\end{align*}
which both as well equals $\int\limits_{\Omega} \Phi(0) \, dx=0$. Moreover, we get $\int\limits_{\Omega} \Phi (1, 0, 0)^T =L (0, x_1, x_2)^T$, which equals
\begin{align*}
\lim\limits_{n \rightarrow \infty} \int\limits_{\Omega} a_n^{11} b_{22}^n + a_n^{12} b_{23}^n + a_n^{21} b_{32}^n + a_n^{22} b_{33}^n \, dx = \lim\limits_{n \rightarrow \infty} \int\limits_{\Omega} a_n^{12} b_{23}^n + a_n^{21} b_{32}^n\, dx,  
\end{align*}
and $\int\limits_{\Omega} \Phi (-1, 0, 0)^T=L (0, x_2, x_1)^T$, which equals
\begin{align*}
\lim\limits_{n \rightarrow \infty} \int\limits_{\Omega} a_n^{11} b_{33}^n + a_n^{12} b_{32}^n + a_n^{21} b_{23}^n + a_n^{22} b_{22}^n \, dx = \lim\limits_{n \rightarrow \infty} \int\limits_{\Omega} a_n^{12} b_{32}^n + a_n^{21} b_{23}^n \, dx.
\end{align*}
Furthermore, for $M \in \R$ we achieve the expressions
\begin{align*}
L \begin{pmatrix} 0 \\ (M+2)x_1+x_2 \\ 0 \end{pmatrix} = \lim\limits_{n \rightarrow \infty} \int\limits_{\Omega} (M+2)^2 a_n^{11} b_{22}^n + (M+2) a_n^{12} b_{22}^n + (M+2) a_n^{21} b_{22}^n +a_n^{22} b_{22}^n \, dx
\end{align*}
and
\begin{align*}
L \begin{pmatrix} 0 \\ 0 \\ (M+1)x_1+x_2 \end{pmatrix} = \lim\limits_{n \rightarrow \infty} \int\limits_{\Omega} (M+1)^2 a_n^{11} b_{33}^n + (M+1) a_n^{12} b_{33}^n + (M+1) a_n^{21} b_{33}^n +a_n^{22} b_{33}^n \, dx
\end{align*}
for $\int\limits_{\Omega} \Phi(0) \, dx=0$. Putting all these things together, we achieve
\begin{align*}
 \int\limits_{\Omega} \Phi \begin{pmatrix} 1 \\ 0 \\ 0 \end{pmatrix} =  \, L \begin{pmatrix} 0 \\ (M+2)x_1+x_2 \\ (M+1)x_1+x_2 \end{pmatrix} = \int\limits_{\Omega} (M+2) \Phi \begin{pmatrix} 1 \\ 0 \\ 0 \end{pmatrix} + (M+1) \Phi \begin{pmatrix} -1 \\ 0 \\ 0 \end{pmatrix} \, dx
\end{align*}
which implies $\Phi (1, 0, 0)^T= -\Phi (-1, 0, 0)^T$, but that is impossible since $\Phi (1, 0, 0)^T>0$ and $\Phi (-1, 0, 0)^T>0$.
\end{proof}

\begin{proof}[Proof of Theorem \ref{corGegenbeispielCart}]
 If there were a sequence $L^n$ $\Gamma(L^2(\Omega,\R^N))$-converging to $L$ and satisfying the conditions of Lemma \ref{indepa}, with Lemma \ref{indepa} we could find another sequence $L^n$ defined by $L^n(u):=\int\limits_{\Omega} a_n^{\alpha \beta} b^n_{i j} u^i_{\alpha}(x) u^j_{\beta}(x) \, dx$ so that $L^n(u)$ $\Gamma(L^2(\Omega,\R^N))$-converges to $L(u)$ for every $u \in W^{1,2}(\Omega, \R^N)$. The rest of the proof is analogous to the proof of Theorem \ref{GegenbeispielCart}.
\end{proof}

In Theorem \ref{counteriso} and Theorem \ref{dom}, we have seen that not every perfect dominance functional of a Cartan functional in $\mathcal{C}(\Omega)$ can be approximated by elements of $\mathcal{R}_I(\Omega)$, no matter if the Cartan functional is even or not. Now we will see that under the assumptions of Chapter \ref{PropApprox} an approximation is not possible even without demanding isotropy and without demanding that one of the Riemannian manifolds is a Euclidean domain.

\begin{satz}\label{GegenbeispielDom}
 Not every perfect dominance functional of a Cartan functional in $\mathcal{C}(\Omega)$ can be approximated by metrics of the form $L^n(u)=\int\limits_{\Omega} a_n^{\alpha \beta}(x) b^n_{i j} (u(x)) u^i_{\alpha}(x) u^j_{\beta}(x) \, dx$ which satisfy the conditions of Lemma \ref{indepb}.
\end{satz}

\begin{korollar}\label{corGegenbeispielDom}
 Not every perfect dominance functional of a Cartan functional in $\mathcal{C}(\Omega)$ can be approximated by metrics of the form $L^n(u)=\int\limits_{\Omega} a_n^{\alpha \beta}(x) b^n_{i j} u^i_{\alpha}(x) u^j_{\beta}(x) \, dx$ which satisfy the conditions of Lemma \ref{indepa}.
\end{korollar}

\begin{bemerkung}
 As in Corollary \ref{corGegenbeispielallg} and \ref{corGegenbeispielCart}, Corollary \ref{corGegenbeispielDom} restricts the approximating metrics $L^n$ to those metrics whose integrands are independent of $u(x)$. Again by this effort, we do not need the condition \eqref{explrecseq} any more. Since this condition is difficult to prove, it is nice to be able to omit it.
\end{bemerkung}

\begin{proof}[Proof of Theorem \ref{GegenbeispielDom}]
 Define the Cartan functional $L$ in terms of its parametric integrand $\Phi(s,z):= 3 \lvert z \rvert = k \cdot \lvert z \rvert +\Phi^*(s,z)$ with $\Phi^*(s,z):=\lvert z \rvert$, $k=2$ and let $L^*$ be the Cartan functional with the parametric integrand $\Phi^*$. Then obviously, $\Phi^*$ satisfies $m_1^* \lvert z \rvert \leq \Phi^*(s,z) \leq m_2^* \lvert z \rvert$ for $m_1^*=m_2^*=1$. Furthermore, we have $\lvert z \rvert \xi^T \Phi^*_{zz}(s,z) \xi = \lvert \xi \rvert^2 - \frac{1}{\lvert z \rvert^2} (\xi \cdot z)^2$ and
\begin{align*}
 \lvert P_{z^{\bot}} \xi \rvert^2 &= \left\lvert \xi - (\xi \cdot z) z \tfrac{1}{\lvert z \rvert^2}  \right\rvert^2 = \lvert \xi \rvert^2 + (\xi \cdot z)^2 \tfrac{1}{\lvert z \rvert^2} - 2 (\xi \cdot z)^2 \tfrac{1}{\lvert z \rvert^2} = \lvert \xi \rvert^2 - (\xi \cdot z)^2 \tfrac{1}{\lvert z \rvert^2} 
\end{align*}
so that we get $\lvert z \rvert \xi^T \Phi^*_{zz}(s,z) \xi \geq \lambda_{L^*}(R_0) \lvert P_{z^{\bot}} \xi \rvert^2$ for every $\xi, s, z \in \R^3$, $\lvert s \rvert \leq R_0$ with $\lambda_{L^*}(R_0)=1$ for all $R_0 > 0$. Thus, $\lambda^*:=\inf\limits_{R_0 \in (0, \infty]} \lambda_{L^*}(R_0)=1$. Moreover, we have $2=k>k_0:=2(m_2^*-\min\{\lambda^*, \tfrac{m_1^*}{2} \} )=1$. Thus, $\Phi$ possesses a perfect dominance function $g(s,A):=\lvert A \rvert^2+g^{* \frac{1}{4}}(s,A)$ with $g^{* \frac{1}{4}}(s,A):=\tfrac{1}{2} \lvert A \rvert^2 ( \tfrac{1}{2} + \tfrac{1}{2} \eta^{\frac{1}{4}}(\tau(A)))$ with $\tau(A)=\frac{2\lvert A_1 \wedge A_2 \rvert}{\lvert A \rvert^2}$ for $A \neq 0$, $\tau(0)=1$ (see \cite[Proofs of Theorems 1.3, 2.8, 2.14]{Hildebrandt}). Here, $\eta^{\frac{1}{4}}: [0,1] \rightarrow \R$ is a cut-off function with $\eta^{\frac{1}{4}}(0)=0$, $\eta^{\frac{1}{4}}(1)=1$ and $0< \eta^{\frac{1}{4}}(r) < 1$ for $u_0^{\frac{1}{4}} < r < 1$ with $u_0^{\frac{1}{4}}=\frac{1}{6}$. We will see that every such function $\eta^{\frac{1}{4}}$ will provide a counterexample. Simple computation yields
\begin{align}
 g (e_1 | e_2+e_3)&=3 +\tfrac{3}{2} (\tfrac{1}{2} + \tfrac{1}{2} \eta^{\frac{1}{4}}(\tfrac{2 \sqrt{2}}{3}) ) =\tfrac{15}{4} + \tfrac{3}{4} \eta^{\frac{1}{4}}(\tfrac{2 \sqrt{2}}{3}) \label{Dom1}
\end{align}
with $0 < \eta^{\frac{1}{4}}(\tfrac{2 \sqrt{2}}{3}) <1$ since $\frac{1}{6} < \frac{2 \sqrt{2}}{3} < 1$ and
\begin{align}
 & g ( e_1 | e_2) +  g(e_1 | e_3) - g (e_1 | 0) + g (0 | e_2+e_3) - g(0 | e_2) - g(0 | e_3) \nonumber\\
=& \, 2+1 ( \tfrac{1}{2} + \tfrac{1}{2} \eta^{\frac{1}{4}}(1)) +2+ 1 ( \tfrac{1}{2} + \tfrac{1}{2} \eta^{\frac{1}{4}}(1)) - (1+\tfrac{1}{2} ( \tfrac{1}{2} + \tfrac{1}{2} \eta^{\frac{1}{4}}(0))) +2 + 1 ( \tfrac{1}{2} + \tfrac{1}{2} \eta^{\frac{1}{4}}(0)) \nonumber\\
&- (1+\tfrac{1}{2} ( \tfrac{1}{2} + \tfrac{1}{2} \eta^{\frac{1}{4}}(0))) - (1+\tfrac{1}{2} ( \tfrac{1}{2} + \tfrac{1}{2} \eta^{\frac{1}{4}}(0)))  = \, \tfrac{19}{4}. \label{Dom2}
\end{align}
Now if there were a sequence $L^n$ $\Gamma(L^2(\Omega,\R^N))$-converging to $G(u)=\int\limits_{\Omega} g(u(x), Du(x)) \, dx$ and satisfying the conditions of Lemma \ref{indepb}, with Lemma \ref{indepa} and Lemma \ref{indepb} we could find another sequence $L^n$ defined by $L^n(u):=\int\limits_{\Omega} a_n^{\alpha \beta} b^n_{i j} u^i_{\alpha}(x) u^j_{\beta}(x) \, dx$ so that $L^n(u)$ $\Gamma(L^2(\Omega,\R^N))$-converges to $G(u)$ for every $u \in W^{1,2}(\Omega, \R^N)$. By Lemma \ref{constrecseq}, we can choose the constant sequence $u$ as recovery sequence. This implies that $G (x_1, x_2, x_2)^T = \int\limits_{\Omega} g (e_1 | e_2+e_3)^T \, dx$ equals
\begin{align*}
\lim\limits_{n \rightarrow \infty} \int\limits_{\Omega} a^{11}_n b_{11}^n+ a^{12}_n b_{12}^n + a^{12}_n b_{13}^n + a^{21}_n b_{21}^n + a^{22}_n b_{22}^n + a^{22}_n b_{23}^n + a^{21}_n b_{31}^n + a^{22}_n b_{32}^n + a^{22}_n b_{33}^n \, dx.
\end{align*}
On the other hand, we would have
\allowdisplaybreaks\begin{align*}
 &G \begin{pmatrix} x_1 \\ x_2 \\ 0 \end{pmatrix} + G \begin{pmatrix} x_1 \\ 0 \\ x_2 \end{pmatrix} - G \begin{pmatrix} x_1 \\ 0 \\ 0 \end{pmatrix} + G \begin{pmatrix} 0 \\ x_2 \\ x_2 \end{pmatrix} - G \begin{pmatrix} 0 \\ x_2 \\ 0 \end{pmatrix} - G \begin{pmatrix} 0 \\ 0 \\ x_2 \end{pmatrix} \\
=& \int\limits_{\Omega} g (e_1 | e_2) +g (e_1 | e_3) - g (e_1| 0) + g (0 | e_2 + e_3) - g (0 | e_2) - g (0 | e_3) \, dx \\
=&  \lim\limits_{n \rightarrow \infty} \int\limits_{\Omega} a^{11}_n b_{11}^n+ a^{12}_n b_{12}^n + a^{12}_n b_{13}^n + a^{21}_n b_{21}^n + a^{22}_n b_{22}^n + a^{22}_n b_{23}^n + a^{21}_n b_{31}^n + a^{22}_n b_{32}^n + a^{22}_n b_{33}^n \, dx,
\end{align*}
which equals $\int\limits_{\Omega} g (e_1 | e_2+e_3) \, dx$ as we have seen above. This implies
\begin{align*}
  g (e_1 | e_2+e_3) = g (e_1 | e_2) +g (e_1 | e_3) - g (e_1| 0) + g (0 | e_2 + e_3) - g (0 | e_2) - g (0 | e_3) 
\end{align*}
which by \eqref{Dom1} and \eqref{Dom2} is equivalent to $\frac{15}{4} + \frac{3}{4} \eta^{\frac{1}{4}}(\tfrac{2 \sqrt{2}}{3})=\tfrac{19}{4}$. Hence, $\eta^{\frac{1}{4}}(\tfrac{2 \sqrt{2}}{3}) = \tfrac{4}{3}>1$ which is a contradiction to $0 < \eta^{\frac{1}{4}}(\tfrac{2 \sqrt{2}}{3}) <1$.
\end{proof}

\begin{proof}[Proof of Corollary \ref{corGegenbeispielDom}]
 Define $g$ as in the proof of Theorem \ref{GegenbeispielDom}. If there were a sequence $L^n$ $\Gamma(L^2(\Omega,\R^N))$-converging to $G(u)=\int\limits_{\Omega} g(u(x), Du(x)) \, dx$ and satisfying the conditions of Lemma \ref{indepa}, with Lemma \ref{indepa} we could find another sequence $L^n$ defined by $L^n(u):=\int\limits_{\Omega} a_n^{\alpha \beta} b^n_{i j} u^i_{\alpha}(x) u^j_{\beta}(x) \, dx$ so that $L^n(u)$ $\Gamma(L^2(\Omega,\R^N))$-converges to $G(u)$ for every $u \in W^{1,2}(\Omega, \R^N)$. The rest of the proof is analogous to the proof of Theorem \ref{GegenbeispielDom}.
\end{proof}

To see that there are perfect dominance functionals of Cartan functionals which can be approximated by metrics satisfying the conditions \eqref{bounded}, \eqref{growth}, \eqref{UC}, \eqref{fusco} and \eqref{constrecseq}, note that the energy functional of $g(A)=\frac{1}{2} \lvert A \rvert^2$ clearly can be approximated by metrics with $a_n^{\alpha \beta}=\frac{1}{2} \delta^{\alpha}_{\beta}$, $b_{ij}^n=\delta_i^j$, where $\delta_i^j$ is the Kronecker delta with value $1$ if $i=j$ and $0$ otherwise. By Corollary \ref{exfus} and Remark \ref{Vererbung}, the approximating metrics satisfy \eqref{explrecseq}, the other conditions of Lemma \ref{indepb} are clearly satisfied. Note that $G(u)=\int\limits_{\Omega} g(Du(x)) \, dx$ is lower semicontinuous with respect to the weak $W^{1,2}(\Omega, \R^3)$-convergence and a perfect dominance functional of the even Cartan functional $L$ with parametric integrand $\Phi(z)=\lvert z \rvert$ (see \cite[p 301]{Hildebrandt}), and that the metrics of the energy functionals $L^n$ in this example are isotropic. There are non-even Cartan functionals whose perfect dominance functionals define energy functionals which can be approximated by metrics satisfying the above conditions as well. Of course, these metrics cannot be isotropic in this case. An example is the Cartan functional with the parametric integrand $\Phi(z)=\lvert z \rvert + z_3$, which possesses the perfect dominance functional $G(u)=\int\limits_{\Omega} g(Du(x)) \, dx$ with $g(A)=\frac{1}{2} \lvert A \rvert^2 + \begin{pmatrix} 0 \\ 0 \\ \frac{1}{2} \end{pmatrix} (A_1 \wedge A_2)$ (see \cite[p 302]{Hildebrandt}). $G$ is lower semicontinuous with respect to the weak $W^{1,2}(\Omega, \R^3)$-convergence, so the respective energy functional can be approximated by metrics with $a_n^{\alpha \alpha}=b_{ii}^n=\frac{1}{\sqrt{2}}$, $a^{12}_n=-a^{21}_n=b_{12}^n=-b_{21}^n=\frac{1}{2}$ since
\begin{align*}
 g(A)=& \, \frac{1}{2} \lvert A \rvert^2 + \begin{pmatrix} 0 \\ 0 \\ \frac{1}{2} \end{pmatrix} (A_1 \wedge A_2) \\
 =\, &\frac{1}{2} (A_1^1)^2 + \frac{1}{2} (A_1^2)^2 + \frac{1}{2} (A_1^3)^2  + \frac{1}{2} (A_2^1)^2 + \frac{1}{2} (A_2^2)^2 + \frac{1}{2} (A_2^3)^2 + \frac{1}{2} (A_3^1)^2 + \frac{1}{2} (A_3^2)^2 + \frac{1}{2} (A_3^3)^2 \\
 &+ \frac{1}{2} A_1^1 A_2^2 - \frac{1}{2} A_1^2 A_2^1
\end{align*}
and
\begin{align*}
 L^n(A)= & \int\limits_{\Omega} \frac{1}{2} (A_1^1)^2 + \frac{1}{2} \frac{1}{\sqrt{2}} A_1^1 A_1^2 -\frac{1}{2} \frac{1}{\sqrt{2}} A^2_1 A^1_1 + \frac{1}{2} (A^2_1)^2 + \frac{1}{2} (A^3_1)^2  \\
 & + \frac{1}{\sqrt{2}} \frac{1}{2} A^1_1 A^1_2 + \frac{1}{4} A^1_1 A^2_2 - \frac{1}{4} A^2_1 A^1_2 + \frac{1}{2} \frac{1}{\sqrt{2}} A_1^2 A_2^2 + \frac{1}{2} \frac{1}{\sqrt{2}} A^3_1 A^3_2 \\
 & - \frac{1}{\sqrt{2}} \frac{1}{2} A^1_2 A^1_1 - \frac{1}{4} A^1_2 A^2_1 + \frac{1}{4} A^2_2 A^1_1 - \frac{1}{2} \frac{1}{\sqrt{2}} A_2^2 A_1^2 - \frac{1}{2} \frac{1}{\sqrt{2}} A^3_2 A^3_1 \\
 & + \frac{1}{2} (A_2^1)^2 + \frac{1}{2} \frac{1}{\sqrt{2}} A_2^1 A_2^2 -\frac{1}{2} \frac{1}{\sqrt{2}} A^2_2 A^1_2 + \frac{1}{2} (A^2_2)^2 + \frac{1}{2} (A^3_2)^2  \, dx \\
 =&  \int\limits_{\Omega} \frac{1}{2} (A_1^1)^2 + \frac{1}{2} (A^2_1)^2 + \frac{1}{2} (A^3_1)^2  + \frac{1}{2} A^1_1 A^2_2 - \frac{1}{2} A^2_1 A^1_2 + + \frac{1}{2} (A_2^1)^2 + \frac{1}{2} (A^2_2)^2 + \frac{1}{2} (A^3_2)^2  \, dx \\
 =&\int\limits_{\Omega} g(A) \,dx.
\end{align*}
Again by Corollary \ref{exfus} the approximating metrics satisfy \eqref{explrecseq} and the other conditions are clearly satisfied.

\bibliographystyle{vancouver}
\bibliography{Literatur}

\section*{Figure captions}

\captionof{mytype}{Partitioning of $\R^m$ into the sets $A_z$, $B_z$ and $Q$}

\section*{Figures}

\begin{center}
\begin{tikzpicture}[scale=0.5]
 \filldraw [color=gray!40] (0,0) rectangle (13,13);
 \draw[-] (0,0) -- (13,0);
 \draw[-] (0,0) -- (0,13);
 \draw [fill=gray!20] (1,1) rectangle (6,6);
 \draw [fill=white] (2,2) rectangle (5,5);
 \draw [fill=gray!20] (7,7) rectangle (12,12);
 \draw [fill=white] (8,8) rectangle (11,11);
 \draw [fill=gray!20] (1,7) rectangle (6,12);
 \draw [fill=white] (2,8) rectangle (5,11);
 \draw [fill=gray!20] (7,1) rectangle (12,6);
 \draw [fill=white] (8,2) rectangle (11,5);
 \fill (2,2) circle (2 pt);
 \node [below] at (2,2) {$\sigma_{z_1}$};
 \fill (8,2) circle (2 pt);
 \node [below] at (8,2) {$\sigma_{z_2}$};
 \fill (2,8) circle (2 pt);
 \node [below] at (2,8) {$\sigma_{z_3}$};
 \fill (8,8) circle (2 pt);
 \node [below] at (8,8) {$\sigma_{z_4}$};
 \fill (1,1) circle (2 pt);
 \node [below] at (1,1) {$\tau_{z_1}$};
 \fill (7,1) circle (2 pt);
 \node [below] at (7,1) {$\tau_{z_2}$};
 \fill (1,7) circle (2 pt);
 \node [below] at (1,7) {$\tau_{z_3}$};
 \fill (7,7) circle (2 pt);
 \node [below] at (7,7) {$\tau_{z_4}$};
 \node at (3.5,3.5) {$B_{z_1}$};
 \node at (9.5,3.5) {$B_{z_2}$};
 \node at (3.5,9.5) {$B_{z_3}$};
 \node at (9.5,9.5) {$B_{z_4}$};
 \node at (3.5,1.5) {$A_{z_1}$};
 \node at (9.5,1.5) {$A_{z_2}$};
 \node at (3.5,7.5) {$A_{z_3}$};
 \node at (9.5,7.5) {$A_{z_4}$};
 \node at (5.5,6.5) {$Q$};
\end{tikzpicture}\label{Bild}
\end{center}

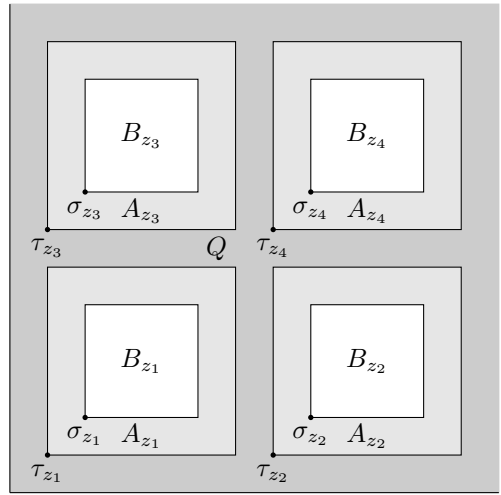
\end{document}